\documentclass[a4paper,12pt]{amsart}
\usepackage[utf8]{inputenc}

\usepackage{hyperref, mathtools, enumerate}
\usepackage{a4wide}
\usepackage[british]{isodate}

\usepackage{amsfonts,amsmath,amsthm,amssymb,mathrsfs}
\usepackage[mathcal]{eucal}

\DeclareMathOperator{\arccoth}{arccoth}
\DeclareMathOperator{\diag}{diag}
\DeclareMathOperator{\ric}{Ric}
\DeclareMathOperator{\hess}{Hess}
\DeclareMathOperator{\tr}{tr}

\newtheorem{thm}{Theorem}[section]
\newtheorem{pro}[thm]{Proposition}
\newtheorem{lem}[thm]{Lemma}
\newtheorem{cor}[thm]{Corollary}

\numberwithin{equation}{section}

\title{Classification of superpotentials for cohomogeneity one Ricci solitons}
\author{Qiu Shi Wang}
\address{Mathematical Institute\\ University of Oxford \\
	Oxford\\
	OX2 6GG\\
	United Kingdom} \email{wangqs@maths.ox.ac.uk}
\cleanlookdateon
\date{14 November 2024}
\thanks{The author is supported by the Engineering and Physical Sciences Research Council [grant number EP/W524311/1] and the Fonds de recherche du Québec -- Nature et technologies [numéro de dossier 331732].}

\begin{document}

\begin{abstract}
	We classify superpotentials for the Hamiltonian system corresponding to the cohomogeneity one gradient Ricci soliton equations. Aside from recovering known examples of superpotentials for steady solitons, we find a new superpotential on a specific case of the Bérard Bergery--Calabi ansatz. The latter is used to obtain an explicit formula for a steady complete soliton with an equidistant family of hypersurfaces given by circle bundles over $S^2\times S^2$. There are no superpotentials in the non-steady case in dimensions greater than 2, even if polynomial coefficients are allowed. We also briefly discuss generalised first integrals and the limitations of some known methods of finding them.
\end{abstract}

\maketitle



\section{Introduction}

The Einstein equations can be significantly simplified when the underlying manifold is assumed to have certain symmetries. For instance, on a homogeneous space $G/K$, all curvature quantities can be computed from the Lie algebra $\mathfrak{g}$ of $G$ and the critical points of the scalar curvature functional can be used to find Einstein manifolds \cite{WangZiller1986}. A natural generalisation of homogeneous spaces is the notion of a cohomogeneity one space, i.e. $G$ acts on the manifold via isometries with codimension one principal orbits. In this context, the Einstein equations simplify to a system of ordinary differential equations (ODEs) in the transverse direction to the principal orbits \cite{Berard-Bergery1982,EschenburgWang2000}. In certain cases, the abovementioned ODEs are explicitly integrable, producing new examples of Einstein manifolds \cite{Berard-Bergery1982}.

In \cite{DancerWang2000}, A. Dancer and M. Wang consider these ODEs as a constrained Hamiltonian system and study the existence of generalised first integrals, which often correspond to previously known integrable cases of the ODEs. In \cite{DancerWang2005}, the same authors study superpotentials, which simplify the Hamiltonian system to a first-order ODE system which is often explicitly integrable. The existence of superpotentials of a certain natural form is subject to many constraints arising from the convex geometry of the weight vectors of the scalar curvature formula of \cite{WangZiller1986}. This allows a full classification of superpotentials in the (most interesting) Ricci-flat case, which is meticulously detailed by Dancer and Wang in \cite{DancerWang2005,DancerWang2008,DancerWang2011b}.

A Ricci soliton is given by a Riemannian manifold $(M,\hat g)$ with a complete vector field $X$ on it such that for some $\epsilon\in\mathbb{R}$,
\begin{equation}\label{solitonequation}
	\ric(\hat g) + \frac{1}{2}\mathscr{L}_X \hat g + \frac{\epsilon}{2}\hat g=0.
\end{equation}
When $\epsilon<0$, $\epsilon=0$ or $\epsilon>0$, the soliton is said to be \textit{shrinking}, \textit{steady} or \textit{expanding} respectively. 
Ricci solitons generate simple solutions to the Ricci flow. They are a generalisation of Einstein metrics, as (\ref{solitonequation}) reduces to the Einstein condition when $X$ is a Killing vector field of $(M,\hat g)$. In the cohomogeneity one context, (\ref{solitonequation}) also simplifies to a system of ODEs \cite{DancerWang2011} and a Hamiltonian approach is also possible \cite{BetancourtDancerWang2016}, with the configuration space extended to include the soliton potential. Generalised first integrals and superpotentials are also defined in \cite{BetancourtDancerWang2016}, and some examples are given, leading to explicit formulae for e.g. the 5-dimensional Bryant soliton. The main objective of the present work is to classify superpotentials for cohomogeneity one solitons using constraints from convex geometry, analogously to the work done in \cite{DancerWang2005,DancerWang2008,DancerWang2011b} for the Ricci-flat condition.

Our main result is Theorem \ref{classification}, which states that the only scalar curvature weight vector configurations admitting a superpotential are the Bérard Bergery--Calabi ansatz and a small number of low-dimensional examples. The result is analogous to Theorems 1.12 and 1.14 of \cite{DancerWang2011b}, which classify superpotentials for the cohomogeneity one Ricci-flat Einstein equations. We also consider the more restrictive non-steady case, which admits no superpotentials with constant coefficients, and just one, with 1-dimensional fibres, when we allow the coefficients to be polynomials.

All superpotentials in the classification had previously been found in \cite{BetancourtDancerWang2016}, except the one given in Theorem \ref{classification}(\ref{newsup}). We integrate the first-order subsystem corresponding to the latter and verify that suitable boundary conditions at the singular orbit are satisfied to ensure smooth extension of the metric. As a result, we obtain explicit formulae for a steady complete soliton on the space foliated by hypersurfaces given by $U(1)$-bundles over $S^2\times S^2$, detailed in Proposition \ref{explicitsoliton}. There is also a brief discussion of generalised first integrals. Thus, the present work is to a certain extent a sequel of \cite{BetancourtDancerWang2016}.

The paper is organised as follows. First, we introduce in Section \ref{notation} the Hamiltonian formalism in question. Section \ref{supersection}, composing the bulk of the paper, details the classification of superpotentials in the steady case. In Section \ref{nonsteadysection}, we discuss the non-steady case and superpotentials with polynomial coefficients. In Section \ref{examplesection}, we integrate the first-order subsystem for the newly found superpotential and write down explicit formulae for the corresponding $6$-dimensional soliton. Finally, in Section \ref{firstintegralsection}, we discuss the limitations of strategies used in \cite{BetancourtDancerWang2016} to find generalised first integrals.

\section{Hamiltonian formalism for cohomogeneity one gradient solitons}\label{notation}

In this section, we present the Ricci soliton equations in the cohomogeneity one setting. We then introduce the Hamiltonian formalism and fix the notation which will be used throughout the rest of the paper.

A \textit{gradient Ricci soliton} on a smooth manifold $M$ consists of a Riemannian metric $\hat{g}$ and a real-valued function $u$ on $M$ such that for some $\epsilon\in\mathbb{R}$,
\begin{equation}\label{gradsolitondef}
	\ric \hat{g} + \hess u + \frac{\epsilon}{2}\hat{g}=0.
\end{equation}
The vector field $X$ in (\ref{solitonequation}) is given by the gradient of $u$.

Let $G$ be a compact Lie group acting on $(M,\hat g)$ via isometries with cohomogeneity one. There is a dense open submanifold $M_0\subset M$ diffeomorphic to $I\times G/K$, where $K\subset G$ is a closed subgroup and $I\subset \mathbb{R}$ is an interval with coordinate $t$. Each $\{t\}\times G/K$ is a \textit{principal orbit}, and we choose a unit speed geodesic $\gamma(t)$ intersecting each principal orbit $\hat g$-orthogonally. Then, a general cohomogeneity one metric on $M_0$ takes the form
\begin{equation*}
	\hat{g}=dt^2 + g_t,
\end{equation*}
where $g_t$ is a 1-parameter family of $G$-invariant metrics on $G/K$. We suppose further that there is precisely one \textit{singular orbit} $G/H$, with $K\subsetneq H$, such that $\hat{g}$ extends to a smooth metric on $M=M_0\cup \{G/H\}$. Let $n$ be the dimension of $G/K$. By averaging over the principal orbits, it suffices to consider $G$-invariant soliton potentials $u=u(t)$.

Let the dot denote differentiation with respect to $t$. In the cohomogeneity one case, the soliton condition (\ref{gradsolitondef}) is equivalent to the ODEs \cite{DancerWang2011}
\begin{align}
	-\tr \dot{L} - \tr (L^2) + \ddot{u} + \frac{\epsilon}{2}&=0\label{NN}\\
	ric-\tr(L)L-\dot{L}+\dot{u}L+\frac{\epsilon}{2}I_n &= 0,\label{XY}
\end{align}
where $ric$ is the Ricci endomorphism $\ric(g_t)(X,Y)=g_t(ric(X),Y)$ and $L$ is the shape operator of the principal orbit, defined by $L(X)=\nabla_X N$, where $\nabla$ is the flat connection of $\hat{g}$ and $N$ is the unit normal with $\nabla_N N = 0$. In fact, we have $L=\frac{1}{2}g_t^{-1}\circ \dot{g}_t$, where $g_t$ is viewed as an endomorphism of the tangent space to the principal orbit.

In the case that $M$ is constructed out of an equidistant family of hypersurfaces which are not necessarily homogeneous spaces, the Ricci soliton equation is (\ref{NN}),(\ref{XY}) with the additional equation
\begin{equation*}
	d(\tr L) - \delta^\Delta L =0,
\end{equation*}
where $\delta^\Delta$ is the codifferential for vector-valued 1-forms on the hypersurfaces.

Given a principal orbit $G/K$, let $\mathfrak{g}$ and $\mathfrak{k}$ denote the Lie algebras of $G$ and $K$ respectively. Fix an $\mathrm{Ad}(K)$-invariant decomposition
\begin{equation*}
	\mathfrak{g} = \mathfrak{k}\oplus \mathfrak{p},
\end{equation*}
so that $\mathfrak{p}$ is identified with the tangent space to a principal orbit. We define a Hamiltonian formalism for the system with configuration space $S^2_+(\mathfrak{p})^K\times \mathbb{R}$, where the first factor is identified via a fixed metric $Q$ on $\mathfrak{p}$ to the space of $G$-invariant Riemannian metrics on $G/K$, and the second factor is the soliton potential. In the monotypic case, we let
\begin{equation*}
	\mathfrak{p}=\mathfrak{p}_1 \oplus\dots\oplus \mathfrak{p}_r
\end{equation*}
be a decomposition of $\mathfrak{p}$ into $r$ pairwise inequivalent $\mathrm{Ad}(K)$-irreducible $Q$-orthogonal summands. Let $d_i = \dim \mathfrak{p}_i$, so that $n=\sum_{i=1}^r d_i$. We write each metric endomorphism $q\in S^2_+(\mathfrak{p})^K$ in exponential coordinates $q_i$ as the diagonal matrix $\diag(e^{q_1} I_{d_1}, \dots, e^{q_r}I_{d_r})$. Let $p_i$ denote the corresponding conjugate momenta. Finally, we write the scalar curvature $S$ of the principal orbit as \cite{WangZiller1986}
\begin{equation}\label{scalarcurvature}
	S=\sum_{w\in \mathcal{W}} A_w e^{w\cdot q}.
\end{equation}
We then define the extended vectors
\begin{equation*}
	\mathbf{d}=(d_1,\dots,d_r,-2),\quad \mathbf{q}= (q_1,\dots,q_r,u), \quad \mathbf{p}=(p_1,\dots, p_r,\phi).
\end{equation*}
We also extend the elements of $\mathcal{W}$ by appending a zero as their $(r+1)$-th component, and define $\tilde{\mathcal{W}}= \mathcal{W} \cup \{\mathbf{0}_{\mathbb{R}^{r+1}}\}$. Consider the bilinear form $J$ given by
\begin{equation}\label{polarisedJ}
	J(\mathbf{p}, \mathbf{p'}) = -\left( \sum_{i=1}^r \frac{p_i{p_i}'}{d_i} + \frac{\phi}{2} \sum_{i=1}^r {p_i}' + \frac{\phi'}{2} \sum_{i=1}^r p_i + \frac{n-1}{4}\phi\phi'\right)
\end{equation}
and the \textit{Hamiltonian}
\begin{equation*}
	\mathcal{H}(\textbf{p},\textbf{q})=e^{-\frac{1}{2}\textbf{d}\cdot\textbf{q}} J(\textbf{p}) - e^{\frac{1}{2}\textbf{d}\cdot\textbf{q}}\left(E-\lambda(n+1-u) + \sum_{w\in\mathcal{W}} A_w e^{\textbf{w}\cdot \textbf{q}}\right),
\end{equation*}
where $\lambda=-\epsilon$ and $E\in\mathbb{R}$ is a parameter. Then it is shown in \cite[Theorem 5]{BetancourtDancerWang2016} that
\begin{thm}\label{hamiltoniantheorem}
	Given a principal orbit $G/K$, consider on the symplectic manifold\\ $(S^2_+(\mathfrak{p})^K\times \mathbb{R})\times(S^2(\mathfrak{p}^*)^K\times \mathbb{R}^*)$ the Hamiltonian $\mathcal{H}$. Then the integral curves of $\mathcal{H}$ lying within $\{\mathcal{H}=0\}$ correspond to solutions to the cohomogeneity one Ricci soliton equations (\ref{NN}) and (\ref{XY}).
\end{thm}

\section{Superpotentials for steady solitons}\label{supersection}

A technique to simplify the Hamiltonian system is to find a \textit{superpotential}, i.e. a scalar function $f$ on the phase space of the Hamiltonian $\mathcal{H}$ such that $\mathcal{H}(\textbf{q}, df(\textbf{q}))=0$. Then the constrained integral curve condition of Theorem \ref{hamiltoniantheorem}, and thus the Ricci soliton equations, simplify to the first-order subsystem (\ref{firstordersubsystem}). For a steady soliton, in terms of the extended vectors, the superpotential condition is equivalent to
\begin{equation}\label{supercondition}
	J(\nabla f, \nabla f) = e^{\mathbf{d}\cdot\mathbf{q}}\left( E + \sum_{w\in \mathcal{W}} A_w e^{\mathbf{w}\cdot \mathbf{q}}\right),
\end{equation}
where $\nabla$ is the gradient with respect to $\mathbf{q}$. We make the ansatz
\begin{equation}\label{superansatz}
	f=\sum_{\mathbf{c}\in\mathcal{C}} f_{\mathbf{c}} e^{\mathbf{c}\cdot\mathbf{q}} \quad \mathrm{with} \quad |\mathcal{C}|<\infty,
\end{equation}
for some constants $f_\textbf{c}\neq 0$. Then (\ref{supercondition}) becomes
\begin{equation}\label{superpotentialcases}
	\sum_{\mathbf{a}+\mathbf{c}=\mathbf{b}} J(\mathbf{a},\mathbf{c}) f_{\mathbf{a}}f_{\mathbf{c}} =
	\begin{cases*}
		A_w & if $\mathbf{b} = \mathbf{d}+\mathbf{w}$  for some $\mathbf{w}\in\mathcal{W}$\\
		E & if $\mathbf{b}=\mathbf{d}$\\
		0 & otherwise.	
	\end{cases*}
\end{equation}
We will assume that the zero vector is not in $\mathcal{C}$, as it would only contribute trivially to the superpotential. We recall three important results from \cite{BetancourtDancerWang2016}.
\begin{lem}\label{lemma11}
	Let $\mathbf{v}, \mathbf{w}\in\mathbb{R}^{r+1}$ be vectors such that $v_{r+1} = w_{r+1} = 0$. Then
	\begin{equation*}
		J(\mathbf{v}+\mathbf{d}, \mathbf{w}+\mathbf{d}) = 1-\sum_{i=1}^r \frac{v_iw_i}{d_i}.
	\end{equation*}
\end{lem}
Two points $\textbf{a},\textbf{c}\in\mathcal{C}$ are said to satisfy the \textit{unique sum condition} if $\textbf{a}+\textbf{c}$ cannot be written as a sum of two elements of $\mathcal{C}$ distinct from $\textbf{a},\textbf{c}$. A vector $\textbf{b}$ is $J$\textit{-null} if $J(\textbf{b},\textbf{b})=0$.
\begin{lem}\label{lemma12}
	(i) $\mathrm{conv}(\mathcal{C})\supset \mathrm{conv}(\frac{1}{2}(\mathbf{d} + \tilde{\mathcal{W}}))$.\\
	(ii) If $\mathbf{a},\mathbf{c}\in\mathcal{C}$ satisfy the unique sum condition, then either $J(\mathbf{a},\mathbf{c})=0$ or $\mathbf{a} + \mathbf{c}\in \mathbf{d}+\tilde{\mathcal{W}}
	$. In particular, if $\mathbf{c}$ is a vertex of $\mathrm{conv}(\mathcal{C})$, then either $\mathbf{c}$ is $J$-null or $\mathbf{c}\in \frac{1}{2}(\mathbf{d}+\tilde{\mathcal{W}})$.
\end{lem}
\begin{pro}\label{prop13}
	If no vertex of $\mathrm{conv}(\mathcal{C})$ is $J$-null, then there are no superpotentials of the form (\ref{superansatz}).
\end{pro}

\subsection{Classification of superpotentials}

Several examples of superpotentials have been found in \cite{BetancourtDancerWang2016}. In this subsection, we obtain a classification result, Theorem \ref{classification}, for superpotentials of the form (\ref{superansatz}) under mild hypotheses on the homogeneous space $G/K$.

The superpotential condition is much more rigid for the Ricci soliton equations than for the Ricci-flat equations for a few reasons. First, Proposition \ref{prop13} states that superpotentials for (non-Ricci-flat) solitons must have a null vector, a result analogous to the case of the Einstein condition with a cosmological constant \cite[Theorem 10.1]{DancerWang2005}. Also, the presence of the ``energy'' term $E$ in the extended scalar curvature formula, together with Lemma \ref{lemma12}(i), implies that superpotentials for solitons must contain vectors which extend beyond the hyperplane of vectors whose components sum to $-1$, which contains all scalar curvature vectors. This additional dimension, together with the natural hypothesis that $\dim\mathrm{conv}(\mathcal{W})=r-1$, make (\ref{superpotentialcases}) considerably more restrictive than in the Ricci-flat case and are key to obtaining a full classification with significantly less effort than was necessary in \cite{DancerWang2005,DancerWang2008,DancerWang2011b}.

Throughout this subsection, we assume that $r\geq 2$. If $r=1$ and $n>1$, then we have $\mathcal{W}=\{(-1)\}$ and the only superpotential is the $d_1=n=4$ Bryant soliton, which is detailed in \cite[Example 15]{BetancourtDancerWang2016}. The $n=1$ case is also discussed there.

First, we begin by characterising the set of $J$-null vectors $\mathbf{v} = (v_1, \dots, v_{r+1})$. Notice that $J$ is negative definite on the hyperplane $\{v_{r+1}=0\}$, so its only null vector on it is $\mathbf{0}$. Assume henceforth that $v_{r+1}\neq 0$. Then $\mathbf{v}$ is $J$-null if and only if
\begin{equation*}
	\left( -\frac{2v_1}{v_{r+1}}, \dots, -\frac{2v_r}{v_{r+1}}, -2\right) =\mathbf{d} + \left( -\frac{2v_1}{v_{r+1}}-d_1, \dots, -\frac{2v_r}{v_{r+1}}-d_r, 0\right)
\end{equation*}
is $J$-null. Applying Lemma \ref{lemma11}, this is equivalent to
\begin{equation*}
	\sum_{i=1}^r \frac{(2v_i + d_i v_{r+1})^2}{d_i} = v_{r+1}^2,
\end{equation*}
which is the equation of a slanted cone in $\mathbb{R}^{r+1}$. We define the affine hyperplane 
\begin{equation*}
	P=\{\mathbf{v}\in\mathbb{R}^{r+1}\:|\: v_{r+1}=-1\} = \left\{\frac{1}{2}(\mathbf{d}+ \mathbf{x})\in\mathbb{R}^{r+1}\: \Big{|} \: x_{r+1}=0\right\}.
\end{equation*}
Our first objective is to show that $\mathcal{C}\subset P$.

\begin{lem}\label{Jdkernel}
	Let $\mathbf{v}=(v_1,\dots,v_{r+1})$. Then $J(\mathbf{d},\mathbf{v})=0$ if and only if $v_{r+1}=0$. In particular, if this holds, then $\mathbf{v}$ cannot be a (nontrivial) $J$-null vertex of $\mathrm{conv}(\mathcal{C})$.
\end{lem}
\begin{proof}
	The result follows from a short calculation using the formula (\ref{polarisedJ}) for $J$.
\end{proof}
From this point on, we will often write ``vertex'' to mean ``vertex of $\mathrm{conv}(\mathcal{C})$'' and ``null'' to mean ``$J$-null''.

\begin{lem}\label{edgeproportional}
	Let $\textbf{a},\textbf{c}$ be two null vertices on the same side of $P$ such that the line segment connecting them is an edge of $\mathrm{conv}(\mathcal{C})$. Then $\textbf{a}$ and $\textbf{c}$ are proportional. This also holds if exactly one of $\textbf{a}$ or $\textbf{c}$ lies in $P$.
\end{lem}
\begin{proof}
Assume that $\textbf{a}+\textbf{c}=\textbf{v}+\textbf{w}$ for some $\textbf{v},\textbf{w}$ different from $\textbf{a},\textbf{c}$. We then have that $\textbf{v},\textbf{w}$ are on the line $\overline{\textbf{ac}}$, since the latter is an edge of $\mathrm{conv}(\mathcal{C})$. Let $\textbf{x}$ be the point on $\overline{\textbf{ac}}$ closest to but distinct from $\textbf{a}$. Since $\textbf{v},\textbf{w}$ exist, $\textbf{x}\neq \textbf{c}$. Therefore $\textbf{a},\textbf{x}$ satisfy the unique sum condition, and since their sum is not in $2P\supset \textbf{d}+\mathcal{\tilde{W}}$, we have $J(\textbf{a},\textbf{x})=0$.
Consider the subspace of $\mathbb{R}^{r+1}$ given by $\ker J(\textbf{a},\cdot)=\{\textbf{v}\: |\: J(\textbf{a},\textbf{v})=0\}$. It contains $\textbf{a}$, and since $J(\textbf{a},\textbf{x})=0$, it also contains $\textbf{x}$ and therefore all of $\overline{\textbf{ac}}$. Thus $J(\textbf{a},\textbf{a})=J(\textbf{a},\textbf{c})=J(\textbf{c},\textbf{c})=0$. This means that the entirety of $\overline{\textbf{ac}}$ is null, which implies that $\textbf{a}$ and $\textbf{c}$ are proportional, since the null locus of $J$ is a cone.

Assume otherwise, i.e. $\textbf{a},\textbf{c}$ satisfy the unique sum condition. Then, since they are both on one side of $P$, by Lemma \ref{lemma12}, $J(\textbf{a},\textbf{c})=0$. Since both $\textbf{a}$ and $\textbf{c}$ are null, we must have that they are proportional to each other for the same reason as above.
\end{proof}

\medskip

\textit{Remark.} We will call the procedure used in the above proof the \textit{closest point argument}. It will appear again a number of times throughout the paper.

\medskip

\begin{cor}\label{atmost2}
	There are at most 2 vertices on each side of $P$.
\end{cor}
\begin{proof}
	Suppose there are 3 or more. Since each of them is connected to every other one by a sequence of edges, Lemma \ref{edgeproportional} implies that they are all proportional and thus collinear, which is impossible.
\end{proof}

Recall that elements $\textbf{w}$ of $\mathcal{W}$, which appear in the formula (\ref{scalarcurvature}) for the scalar curvature of $G/K$, take three forms \cite{DancerWang2005,WangZiller1986}.
\begin{enumerate}
	\item Type I: one entry of $\textbf{w}$ is $-1$ and the others are zero. These are denoted $(-1^i)$ for some $i=1,\dots,r$.
	\item Type II: two entries are $-1$, one is $1$ and the others are zero. These are denoted $(1^i,-1^j,-1^k)$ for some $i,j,k=1,\dots,r$.
	\item Type III: one entry is $1$, another is $-2$ and the others are zero. These are denoted $(1^i,-2^j)$ for some $i,j=1,\dots,r$.
\end{enumerate}
For convenience, we will sometimes write any vector in Euclidean space, not just elements of $\mathcal{W}$, as $(a_1^{i_1},a_2^{i_2},\dots,a_m^{i_m})$, where $a_1, \dots, a_m$ are the nonzero entries of the vector, which are respectively in the positions $i_1,\dots, i_m$. We will also abuse notation and systematically identify vectors $\textbf{x}\in\mathbb{R}^r$ with their extensions by zero $(\textbf{x},0)\in\mathbb{R}^{r+1}$.

For a vector $\frac{1}{2}(\textbf{d}+\textbf{w})\in P$, define $s\left(\frac{1}{2}(\textbf{d}+\textbf{w})\right) =\sum_{i=1}^r w_i.$ Then $s(\frac{1}{2}(\textbf{d}+\textbf{w}))=-1$ for each $\textbf{w}\in\mathcal{W}$. We will denote the subset of $P$ in which $s\left(\frac{1}{2}(\textbf{d}+\textbf{w})\right)=\lambda$ by $\{s=\lambda\}$.

In our classification of superpotentials, we will assume that $\mathrm{conv}(\mathcal{W})$ has nonzero $(r-1)$-dimensional measure. This is the same hypothesis as in the main theorem of \cite{DancerWang2008}. It holds, for instance, if $G$ is semisimple. In the case that the isotropy group $K$ is connected, this is not so restrictive, due to the following result.

\begin{lem}\label{fullmeasure}
	Suppose that the isotropy group $K$ is connected and that for each coordinate position in $\mathbb{R}^r$, there is an element of $\mathcal{W}$ with a nonzero entry in that position. Then $\mathrm{conv}(\mathcal{W})$ has nonzero $(r-1)$-dimensional measure.
\end{lem}
\begin{proof}
	We begin by noting that there is at least one type I vector, say $(-1^r)\in\mathcal{W}$ up to permutation of indices. Let $\textbf{w}_1, \dots, \textbf{w}_m$ be the other vectors of $\mathcal{W}$.
	First, we note that it suffices to show that the image of $\mathrm{conv}(\mathcal{W})$ under the affine transformation which sends the last ($r$-th) coordinate to zero has nonzero measure as a subset of $\{x_r=0\}\cong\mathbb{R}^{r-1}$. Since $(-1^r)$ is sent to $\textbf{0}$ by this transformation, it suffices to show that the matrix whose rows are the components of $\textbf{w}_1, \dots, \textbf{w}_m$ has rank $r-1$.
	
	If $(-1^i)\in\mathcal{W}$ for each $i=1, \dots, r-1$, then we are done. If $(-1^k)\notin\mathcal{W}$, then $d_k=1$. Since by hypothesis there exists a vector in $\mathcal{W}$ with a nonzero $k$-th component, by Lemma 4.1 of \cite{DancerWang2005}, it must be a type III vector of the form $(1^k,-2^l)$ for some $l$. Then $d_l\neq 1$, so $(-1^l)\in\mathcal{W}$. It is evident that we may replace $(1^k,-2^l)$ by $(-1^k)$ and the rank of the matrix will stay the same. Doing this for all indices $k$ for which $(-1^k)\notin\mathcal{W}$, we get that the rank of the original matrix is equal to $r-1$.
\end{proof}

From this point onwards, we will \textbf{assume that $K$ is connected and for each coordinate position in $\mathbb{R}^r$, there is an element of $\mathcal{W}$ with a nonzero entry in that position}. For simplicity, this will not be explicitly stated, except in the final result.

\medskip

\textit{Remark.} If there is a superpotential for $G/K$, then we can always add a trivial 1-dimensional factor, extend all weight vectors by a $0$, and obtain a superpotential for the local product of $G/K$ with a circle.

Conversely, suppose there is no element of $\mathcal{W}$ which is nonvanishing in a particular entry, say corresponding to the summand $\mathfrak{p}_1$. Then, in particular, $(-1^1)\notin\mathcal{W}$, and thus $\mathfrak{p}_1$ is a 1-dimensional abelian subalgebra satisfying $[\mathfrak{k},\mathfrak{p}_1]=0$ and $[\mathfrak{p}_1,\mathfrak{p}_j]\subset \mathfrak{p}_j$ for each $j\neq 1$ \cite{DancerWang2005}.

In fact, $\mathfrak{k}\oplus \mathfrak{p}_1$ is a Lie subalgebra of $\mathfrak{g}$. Consequently, there is a unique connected Lie subgroup $H$ of $G$ with Lie algebra $\mathfrak{k}\oplus \mathfrak{p}_1$. The projection map $G/K \rightarrow G/H$ taking $gK$ to $gH$ gives the bundle
\begin{equation*}
	S^1 \rightarrow G/K \rightarrow G/H.
\end{equation*}
The tangent space to $G/K$ at a point is identified with $\mathfrak{p}$. With the bundle structure, the vertical subspace is given by $\mathfrak{p}_1$. We may choose the horizontal distribution given by the complement $\mathfrak{p}_2\oplus \dots\oplus \mathfrak{p}_r$ to define a connection on $G/K$. The curvature of the connection is given by the vertical component of Lie brackets of horizontal vector fields. Using the notation of \cite{WangZiller1986}, we note that since there are no type II or III vectors with ``1'' in the first entry, we must have $\begin{bmatrix*}1\\ i\, j\end{bmatrix*}=0$ for each $i,j\geq 2$. Thus $Q([\mathfrak{p}_i,\mathfrak{p}_j],\mathfrak{p}_1)=0$, and consequently $[\mathfrak{p}_i,\mathfrak{p}_j]$ has zero component in the horizontal direction $\mathfrak{p}_1$, so the connection has zero curvature. To summarise, the hypothesis of no trivial summands, equivalently $\dim\mathrm{conv}(\mathcal{W})=r-1$ by Lemma \ref{fullmeasure}, is natural in the following sense:
\begin{thm}
	Suppose $K$ is connected and there is an entry of $\mathbb{R}^r$ for which all elements of $\mathcal{W}$ vanish. Then there is a subgroup $H\subset G$ such that the bundle $S^1\rightarrow G/K \rightarrow G/H$ is flat. In particular, if $G/K$ is simply connected, then $G/K\simeq S^1\times G/H$.
\end{thm}
We may now continue the classification using this hypothesis.
\begin{lem}\label{bothsidesP}
	If $\mathcal{C}\not\subset P$, then there are vertices on both sides of $P$.
\end{lem}

\begin{proof}
	Assume for sake of contradiction that there are vertices on only one side of $P$, then there are either one or two on that side. Suppose that there is just one, denoted by $\textbf{c}$. Then it is adjacent to every vertex of $\mathrm{conv}(\mathcal{C})$ in $P$, which span an $r$-dimensional subset of $P$. However, they must all lie in $\mathrm{ker} J(\textbf{c},\cdot)\cap P\subsetneq P$, which is a contradiction. To see this, use the closest point argument on the edge connecting each such vertex with $\textbf{c}$, and use the fact that $\textbf{c}$ is null.
	Suppose that there are two, then by Lemma \ref{edgeproportional} they are given by $\textbf{c}$ and $\lambda\textbf{c}$ for some $\lambda\in\mathbb{R}$. Then each of the vertices in $P$ must be adjacent to at least one of these, and thus be in $\mathrm{ker} J(\textbf{c},\cdot)\cap P\subsetneq P$, a contradiction for the same reason as above.
\end{proof}
\begin{lem}\label{CinP}
	$\mathcal{C}\subset P$.
\end{lem}
\begin{proof}
	Suppose otherwise, then by Corollary \ref{atmost2} and Lemma \ref{bothsidesP}, there must be 1 or 2 vertices on either side of $P$. 
		
	Suppose there was a null vertex in $P$. Then it shares an edge with a (null) vertex on either side of $P$. By Lemma \ref{edgeproportional}, these three vertices are proportional and thus collinear, which is impossible.
	
	Assume that none of the $1,2$ or $4$ pairs of vertices with one on each side of $P$ forms an edge of $\mathrm{conv}(\mathcal{C})$. Then $\mathrm{conv}(\mathcal{C}\cap P)=\mathrm{conv}(\mathcal{C})\cap P$.

	There is no null vertex in $P$, so since $\mathrm{conv}(\frac{1}{2}(\textbf{d}+\tilde{\mathcal{W}}))\subset \mathrm{conv}(\mathcal{C})\cap P$ and every vertex outside of $\frac{1}{2}(\textbf{d}+\tilde{\mathcal{W}})$ is null by Lemma \ref{lemma12}, we have that $\mathrm{conv}(\frac{1}{2}(\textbf{d}+\tilde{\mathcal{W}}))=\mathrm{conv}(\mathcal{C})\cap P=\mathrm{conv}(\mathcal{C}\cap P)$, and that every vertex of $\mathrm{conv}(\frac{1}{2}(\textbf{d}+\tilde{\mathcal{W}}))$ is a vertex of $\mathcal{C}$. From this point on, we may follow the proof of Proposition 13 in \cite{BetancourtDancerWang2016}, which we reproduce here for completeness. Since $\textbf{0}$ is a vertex of $\mathrm{conv}(\tilde{\mathcal{W}})$ and there are no null vertices in $P$, there is a $\textbf{w}\in\mathcal{W}$ such that the vertices $\textbf{c}_0=\frac{1}{2}\textbf{d}$ and $\textbf{c}_1=\frac{1}{2}(\textbf{d}+\textbf{w})$ form an edge of $\mathrm{conv}(\mathcal{C})$. Then $J(\textbf{c}_0,\textbf{c}_0)=\frac{1}{4}$ and $J(\textbf{c}_0,\textbf{c}_1)=\frac{1}{4}$, and thus $J(\textbf{c}_0,\cdot)$ is constant and nonzero on the line segment $\overline{\textbf{c}_0\textbf{c}_1}$. Consider the element of $\mathcal{C}$ on $\overline{\textbf{c}_0\textbf{c}_1}$ which is closest to $\textbf{c}_0$. It satisfies the unique sum condition with $\textbf{c}_0$, which contradicts Lemma \ref{lemma12}.
	
	It only remains to check the case where there exists a pair of vertices $\textbf{a}$ and $\textbf{c}$, one on each side of $P$, such that $\overline{\textbf{a}\textbf{c}}$ is an edge of $\mathrm{conv}(\mathcal{C})$. Then either $\textbf{a}+\textbf{c}\in\textbf{d}+\tilde{\mathcal{W}}$ or $\textbf{a}$ and $\textbf{c}$ are proportional, as can be seen as follows.
	
	Let $\textbf{x}_a$ be the element on $\overline{\textbf{a}\textbf{c}}$ which is closest to but not equal to $\textbf{a}$. If $\textbf{x}_a=\textbf{c}$, then $\textbf{a},\textbf{c}$ satisfy the unique sum condition and Lemma \ref{lemma12}, together with the fact that $\textbf{a},\textbf{c}$ are null, implies the claim.
	
	If $\textbf{x}_a\neq \textbf{c}$, we note that $\textbf{a},\textbf{x}_a$ satisfy the unique sum condition. Similarly, we may define $\textbf{x}_c\in \overline{\textbf{a}\textbf{c}}$ such that $\textbf{c},\textbf{x}_c$ satisfy the unique sum condition. Since $\textbf{a}$ and $\textbf{c}$ lie on different sides of $P$, we have that at least one of $\frac{1}{2}(\textbf{a}+\textbf{x}_a)$ or $\frac{1}{2}(\textbf{c}+\textbf{x}_c)$ is not in $\frac{1}{2}(\textbf{d}+\tilde{\mathcal{W}})\subset P$; assume without loss of generality that it is $\frac{1}{2}(\textbf{a}+\textbf{x}_a)$. Then, by Lemma \ref{lemma12}, $J(\textbf{a},\textbf{x}_a)=0$, and by linearity $J(\textbf{a},\textbf{c})=0$, implying that $\textbf{a}$ and $\textbf{c}$ are proportional.

	In the case that $\textbf{a}+\textbf{c}\in\textbf{d}+\tilde{\mathcal{W}}$, then in fact $\textbf{a}+\textbf{c}=\textbf{d}$. To see this, note that if $\textbf{a}+\textbf{c}\in\textbf{d}+\mathcal{W}$, then $\frac{1}{2}\textbf{d}$ is a vertex since it lies outside of the $r$-dimensional affine hyperplane $\{s=-1\}+\overline{\textbf{a}\textbf{c}}$. Consider the convex set $\mathrm{conv}(\mathcal{C})\cap P$. It is clear that there exists a point $\textbf{x}\in P$ such that $\mathrm{conv}(\mathcal{C})\cap P = \mathrm{conv}((\mathcal{C}\cap P) \cup \{\textbf{x}\})$, and in fact one sees that $\{\textbf{x}\}=\overline{\textbf{a}\textbf{c}}\cap P$. Since $r\geq 2$, the vertices of $\mathrm{conv}(\mathcal{C})$ in $P$ are all elements of $\frac{1}{2}(\textbf{d}+\tilde{\mathcal{W}})$ and $\frac{1}{2}\textbf{d}$ has $\geq 2$ adjacent vertices in $\mathrm{conv}((\mathcal{C}\cap P) \cup \{\textbf{x}\})$, with at most one of them being $\textbf{x}$, so that there exists a $\textbf{w}\in\mathcal{W}$ such that $\frac{1}{2}(\textbf{d}+\textbf{w})$ shares an edge with $\frac{1}{2}\textbf{d}$ in $\mathrm{conv}(\mathcal{C})\cap P$. The two edges and vertices remain when we pass to $\mathrm{conv}(\mathcal{C})$, which leads to a contradiction for the same reason as in the proof of Proposition 13 of \cite{BetancourtDancerWang2016}. Hence, $\textbf{a}+\textbf{c}=\textbf{d}$.
	
	Let $\textbf{f}=\frac{1}{2}(\textbf{d}+\textbf{w})$ be a non-null vertex which shares an edge with both $\textbf{a}$ and $\textbf{c}$. Let $\textbf{x}_a$ denote the point of $\mathcal{C}$ on $\overline{\textbf{a}\textbf{f}}$ which is closest to $\textbf{a}$, and let $\textbf{x}_c$ denote the point of $\mathcal{C}$ on $\overline{\textbf{c}\textbf{f}}$ which is closest to $\textbf{c}$. Then we must have by Lemma \ref{lemma12} that $J(\textbf{a},\textbf{x}_a)=0$ and $J(\textbf{c},\textbf{x}_c)=0$. Since $J(\textbf{a},\cdot)$ and $J(\textbf{c},\cdot)$ are affine functions on the abovementioned edges and both $\textbf{a}$ and $\textbf{c}$ are null vertices, we must have $J(\textbf{a},\textbf{f})=J(\textbf{c},\textbf{f})=0$. This is impossible, as $J(\textbf{a},\textbf{f})+J(\textbf{c},\textbf{f})=J(\textbf{d},\frac{1}{2}(\textbf{d}+\textbf{w}))=\frac{1}{2}$. Therefore $\textbf{a}+\textbf{c}\notin \textbf{d}+\tilde{\mathcal{W}}$.
	
	In the final case, namely that $\textbf{a}$ and $\textbf{c}$ are proportional, then by Lemma \ref{bothsidesP} there is only one vertex on each side of $P$. Furthermore, $\frac{1}{2}\textbf{d}$ is not a vertex, since otherwise there must be a non-null vertex sharing an edge with $\frac{1}{2}\textbf{d}$, which is impossible.
	
	Let $\ell=\{\lambda \textbf{a}\: | \: \lambda\in\mathbb{R}\}$ be the (null) line passing through $\textbf{a}$ and $\textbf{c}$, intersecting $P$ at $\mathbf{l}$. Then $\textbf{l}\in\{s\geq 0\}$. Since all vertices of $\mathrm{conv}(\mathcal{C})\cap P$ except possibly $\textbf{l}$ are in $\{s=-1\}$, but $\frac{1}{2}\textbf{d}\in\mathrm{conv}(\mathcal{C})\cap P$, we have that $\textbf{l}\in\{s>0\}$.
	
	Note that $\mathrm{conv}(\mathcal{C})$ has nonzero $(r+1)$-dimensional measure, since it contains the line $\overline{\textbf{a}\textbf{c}}$ through the origin as well as $\mathrm{conv}(\tilde{\mathcal{W}})$, which is an $r$-dimensional subset of $P$. Since all vertices of $\mathrm{conv}(\mathcal{C})$ other than $\textbf{a}$ and $\textbf{c}$ are in $\{s=-1\}\cap P$, which is $(r-1)$-dimensional, there is a set $\mathcal{V}=\{\textbf{b}_i\}_{i=1}^r\subset\{s=-1\}\cap P $ of vertices adjacent to $\textbf{a}$ such that $\dim\mathrm{conv}(\mathcal{V})=r-1$. Applying the closest point argument, we get $J(\textbf{a},\textbf{b}_i)=0$ for each $i$. The $(r-1)$-dimensional affine subspace $\ker J(\textbf{a},\cdot)\cap P$ contains both $\mathcal{V}$ and $\textbf{l}\in\overline{\textbf{a}\textbf{c}}\cap P$, but $\textbf{l}\notin \{s=-1\}$, which is a contradiction.
	\end{proof}

From this point on, in the light of Lemma \ref{CinP}, we will implicitly consider all subsets of $\mathbb{R}^{r+1}$ to be subsets of $P$.

\begin{lem}\label{notd/2}
	$\frac{1}{2}\textbf{d}$ is not a vertex.
\end{lem}
\begin{proof} There cannot be a non-null vertex adjacent to $\frac{1}{2}\textbf{d}$. Let $\textbf{a}$ be a vertex adjacent to $\frac{1}{2}\textbf{d}$, then $J(\textbf{a},\textbf{a})=0$ and $J(\frac{1}{2}\textbf{d},\textbf{a})=\frac{1}{4}$, so that by the closest point argument and Lemma \ref{lemma12} we must have that the line segment between them is empty and that $\textbf{a}=\frac{1}{2}\textbf{d}+\textbf{w}$ for some $\textbf{w}\in\mathcal{W}$. All such vertices $\textbf{a}$, of which in all nontrivial cases $r\geq 2$ there exists more than one, therefore lie in $\{s=-2\}$. Consequently, there are at least two vertices in the half-plane $\{s<-1\}$, which are connected by edges lying entirely in $\{s<-1\}$. Thus, there is least one pair of adjacent (null) vertices in $\{s<-1\}$, which is impossible by the closest point argument and Lemma \ref{lemma12}.
\end{proof}

\begin{lem}\label{abrange}
	Let $a$ and $b$ be the number of (necessarily null) vertices in $\{s>-1\}$ and $\{s<-1\}$ respectively. Then $a\in\{1,2\}$ and $b\in\{0,1\}$.
\end{lem}
\begin{proof}
	Clearly $a\geq 1$, since $\frac{1}{2}\textbf{d}\in\{s\geq 0\}\cap\mathrm{conv}(\mathcal{C})$.
	
	Let $\textbf{a}_1,\textbf{a}_2$ be (null) vertices in $\{s>-1\}$. Let $\textbf{x}_1, \textbf{x}_2$ be the points on $\overline{\textbf{a}_1\textbf{a}_2}$ closest to but not equal to $\textbf{a}_1, \textbf{a}_2$ respectively. Then both $\textbf{a}_1+\textbf{x}_1$ and $\textbf{a}_2+\textbf{x}_2$ are in $\textbf{d}+\tilde{\mathcal{W}}$. However, there is only one point of $\frac{1}{2}(\textbf{d}+\tilde{\mathcal{W}})$ outside of $\{s=-1\}$, so this is impossible unless $\textbf{x}_1=\textbf{a}_2$ and $\textbf{x}_2=\textbf{a}_1$, in which case $\textbf{a}_1,\textbf{a}_2$ satisfy the unique sum condition. Lemma \ref{lemma12} yields that either $J(\textbf{a}_1,\textbf{a}_2)=0$, which implies that $\textbf{a}_1,\textbf{a}_2$ are collinear, contradicting Lemma \ref{CinP}, or $\textbf{a}_1+\textbf{a}_2=\textbf{d}$. Since their pairwise sums are all the same, there cannot be $\geq 3$ vertices in $\{s>-1\}$, so $a\in\{1,2\}$.
	
	If there are 2 vertices in $\{s<-1\}$, then by the same argument they must average to an element of $\frac{1}{2}(\textbf{d}+\tilde{\mathcal{W}})$ in $\{s<-1\}$, which is impossible. Therefore $b\leq 1$.
\end{proof}

\begin{lem}\label{20}
	There are no superpotentials such that $(a,b)=(2,0)$.
\end{lem}

\begin{proof}
	The two vertices in $\{s>-1\}$ must be $\frac{1}{2}(\textbf{d}\pm \textbf{v})$ for some $\textbf{v}$. Furthermore, there cannot be a null vertex in $P$, since its sum with at least one of $\frac{1}{2}(\textbf{d}\pm \textbf{v})$ must be in $\textbf{d}+\tilde{\mathcal{W}}$, which is impossible.
	
	There cannot be a vertex $\frac{1}{2}(\textbf{d}+\textbf{w})$ which is adjacent to both $\frac{1}{2}(\textbf{d}+\textbf{v})$ and $\frac{1}{2}(\textbf{d}-\textbf{v})$, as can be seen by using the closest point argument. However, each vertex is adjacent to one of them, so either $J(\textbf{d}+\textbf{w},\textbf{d}+\textbf{v})=0$, or $J(\textbf{d}+\textbf{w},\textbf{d}-\textbf{v})=0$.
	
	More explicitly, we are saying that all vertices of $\mathrm{conv}(\frac{1}{2}(\textbf{d}+\mathcal{W}))=\mathrm{conv}(\mathcal{C})\cap \{s=-1\}$ must take the form $\frac{1}{2}(\textbf{d}+\textbf{x})$, where
	\begin{equation}\label{parallelplanes}
		\sum_{i=1}^r \frac{v_ix_i}{d_i} = \pm 1.
	\end{equation}
	Furthermore, the two vertices $\frac{1}{2}(\textbf{d}\pm \textbf{v})$ must be null, which gives
	\begin{equation}\label{vnull}
		\sum_{i=1}^r \frac{v_i^2}{d_i}=1.
	\end{equation}
	These two equations together imply that type I vectors cannot be vertices of $\mathrm{conv}(\mathcal{W})$.

	Suppose that there were a type II vertex $(1^1,-1^2,-1^3)$ of $\mathrm{conv}(\mathcal{W})$. Then 
	\begin{equation*}
		\frac{v_1}{d_1} - \frac{v_2}{d_2} - \frac{v_3}{d_3} = \pm 1, \qquad \frac{v_1^2}{d_1} + \frac{v_2^2}{d_2} + \frac{v_3^2}{d_3}\leq 1, 
	\end{equation*}
	with $d_1,d_2,d_3\geq 2$. Suppose that another type II vector of the ``triplet'', say $(-1^1,1^2,-1^3)$, is also a vertex of $\mathrm{conv}(\mathcal{W})$. Then
	\begin{equation*}
		-\frac{v_1}{d_1} + \frac{v_2}{d_2} - \frac{v_3}{d_3} = \pm 1,
	\end{equation*}
	which leads to
	\begin{equation*}
		-\frac{2v_3}{d_3}=\begin{cases*}
			0\\
			\pm 2
		\end{cases*}.
	\end{equation*}
	The case $v_3=\pm d_3$ leads to a contradiction. If $v_3=0$, then 
	\begin{equation*}
		\frac{v_1}{d_1}-\frac{v_2}{d_2}=\pm 1, \qquad \frac{v_1^2}{d_1} + \frac{v_2^2}{d_2}\leq 1
	\end{equation*}
	together imply that $v_1=\pm 1$, $v_2=\pm 1$, $d_1=d_2=2$, and also that $\textbf{v}=\pm(1^1,-1^2)$. This implies that all vertices of $\mathrm{conv}(\mathcal{W})$ lie in the two planes
	\begin{equation*}
		\frac{x_1}{2}-\frac{x_2}{2}=\pm 1,
	\end{equation*}
	and thus are either type III vectors with the $-2$ in the first or second factor, or type II vectors $(1^1,-1^2,-1^i)$ or $(-1^1,1^2,-1^i)$. In particular, we note that $(-1^1,-1^2,1^3)$ cannot be a vertex, so we must have both $\textbf{w}_1=(-2^2,1^3)$ and $\textbf{w}_2=(-2^1,1^3)$ as vertices. Thus, by Theorem 3.5 of \cite{DancerWang2005}, $0=J(\textbf{d}+(-2^2,1^3), \textbf{d}+(1^1,-1^2,-1^3)) = 1/d_3$, which is a contradiction.
	
	Therefore $(-1^1,1^2,-1^3), (-1^1,-1^2,1^3)$ are not vertices of $\mathrm{conv}(\mathcal{W})$. Then $(-2^1,1^3)$ and $(-2^2,1^3)$ are adjacent vertices, so $J(\textbf{d}+(-2^1,1^3), \textbf{d}+(-2^2,1^3))=0$ and thus $d_3=1$, again a contradiction.
	
	Thus, there can only be type III vertices in $\mathrm{conv}(\mathcal{W})$. In fact, there are no type II vectors in $\mathcal{W}$. To see this, assume for sake of contradiction that $(1^1,-1^2,-1^3)\in\mathcal{W}$. Then, by Proposition 4.1(i) and (vii) of \cite{DancerWang2005},  $(1^1,-2^2),(-2^1,1^2),(1^1,-2^3),(-2^1,1^3),(1^2,-2^3),\\(-2^2,1^3)\in\mathcal{W}$. As a result, $d_1=d_2=d_3=4$, and (\ref{parallelplanes}),(\ref{vnull}) require that for each $i,j=1,2,3$ with $i\neq j$,
	\begin{equation*}
		-\frac{v_i}{2} + \frac{v_j}{4} = \pm 1, \qquad \frac{v_i}{4}-\frac{v_j}{2}=\pm 1, \qquad \frac{v_1^2}{4} + \frac{v_2^2}{4} + \frac{v_3^2}{4}\leq 1.
	\end{equation*}
	One checks explicitly that no solutions exist, so there are no type II vectors in $\mathcal{W}$.
	
	Suppose that $(-2^1,1^2)$ is a vertex. Since $(1^i,-2^j)$ and $(1^j,-2^k)$ cannot be adjacent vertices if $i\neq k$, we obtain by ``chasing around'' Figure \ref{diagram} that $(1^2,-2^3)$ is a vertex, and either $(-2^2,1^3)$ or $(-2^2,1^1)$ is. By Theorem 3.5 of \cite{DancerWang2005}, $d_2=1$, which is a contradiction. Therefore there are no type III vectors in $\mathrm{conv}(\mathcal{W})$, and the proof is complete.
\end{proof}

\begin{figure}	
	\hspace{1cm}
	\begin{center}
		\vspace{-.3in}
		\includegraphics[width=0.7\textwidth]{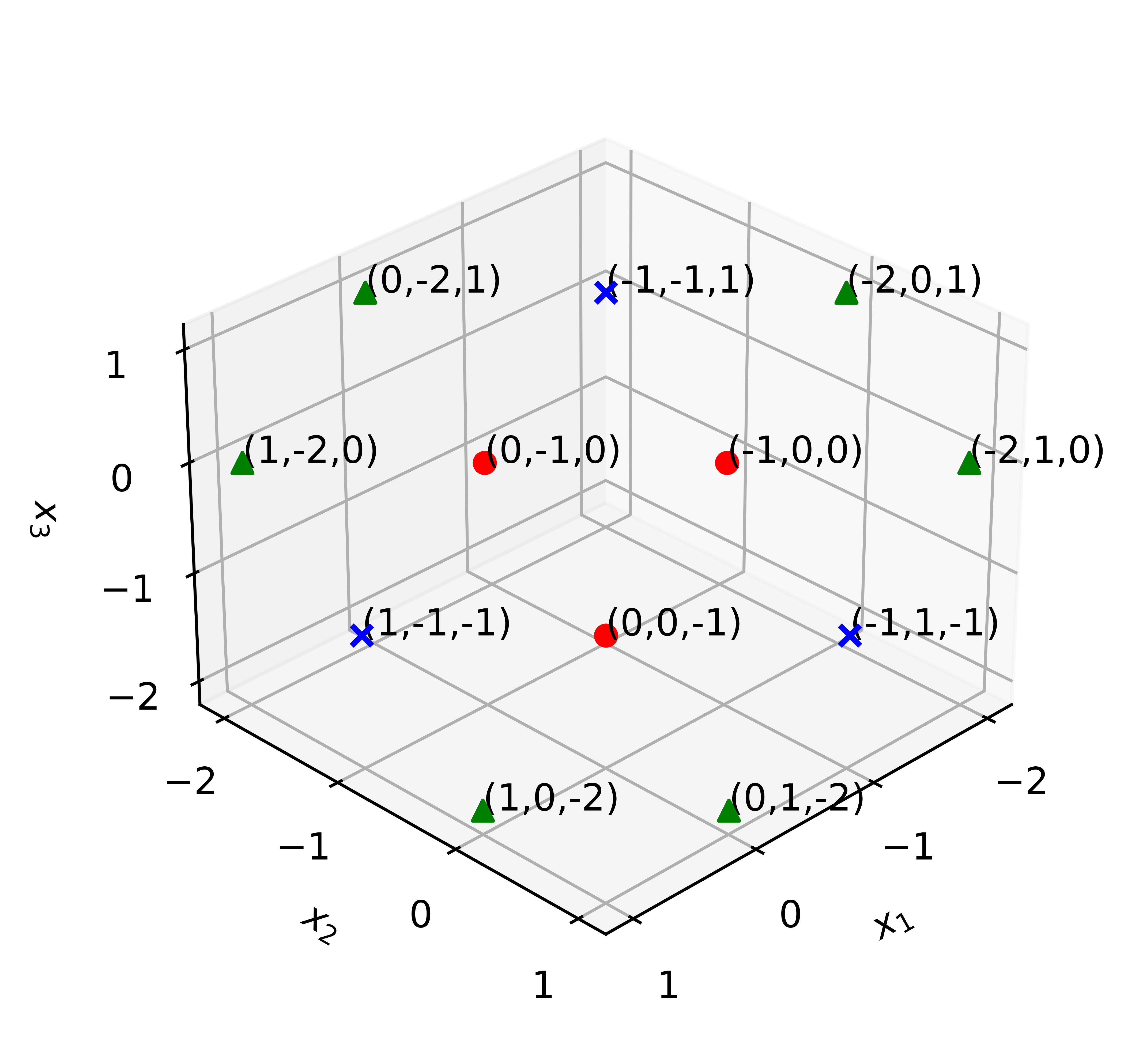}
		\vspace{-.2in}
	\end{center}
	\caption{Plot of the possible type I (red circles), type II (blue crosses) and type III (green triangles) weight vectors of $\mathcal{W}$ in the plane spanned by the first three coordinates.}
	\label{diagram}
\end{figure}

\begin{lem}\label{11}
	There are no superpotentials such that $(a,b)=(1,1)$.
\end{lem}

\begin{proof}
	Denote the vertices in $\{s>-1\}$ and $\{s<-1\}$ by $\textbf{c}_+$ and $\textbf{c}_-$ respectively. Suppose that $\textbf{c}_+,\textbf{c}_-$ are not adjacent. Then $\textbf{c}_+$ is adjacent to an $(r-1)$-dimensional set of non-null vertices in $\{s=-1\}$. But $\textbf{c}_+\in\ker J(\textbf{c}_+,\cdot)$, so $P\subset \ker J(\textbf{c}_+,\cdot)$, which is impossible. Thus $\textbf{c}_+,\textbf{c}_-$ are adjacent. Using the closest point argument, we conclude that $\textbf{c}_++\textbf{c}_-=\textbf{d}+\textbf{w}_0$ for some $\textbf{w}_0\in\tilde{\mathcal{W}}$.
	
	Consider a non-null vertex $\frac{1}{2}(\textbf{d}+\textbf{w})$ adjacent to both $\textbf{c}_\pm$. Then $J(\textbf{d}+\textbf{w},\textbf{c}_\pm)=0$, which add up to $J(\textbf{d}+\textbf{w},\textbf{d}+\textbf{w}_0)=0$. Therefore $\textbf{w}_0\neq \textbf{0}$, and thus $\textbf{c}_+\in\{s>0\}$ and $\textbf{c}_-\in \{s<-2\}$.
	
	For each vertex $\textbf{w}'$ of $\mathrm{conv}(\mathcal{W})$,  $\frac{1}{2}(\textbf{d}+\textbf{w}')$ is adjacent to both $\textbf{c}_\pm$, so we have by the closest point argument that $J(\textbf{c}_\pm, \textbf{d}+\textbf{w}')=0$ and therefore
	\begin{equation}\label{w0product}
		1=\sum_{i=1}^r \frac{w_{0i}w'_i}{d_i}.
	\end{equation}

	Suppose that $\textbf{w}_0$ is of type I, say $\textbf{w}_0=(-1^1)$. Then (\ref{w0product}) becomes $w'_1 = -d_1$. Consequently, $d_1=2$, $w'_1=-2$ and $\textbf{w}'=(-2^1,1^i)$. Since $\dim \mathrm{conv}(\mathcal{W})=r-1$, we must have $\textbf{w}'_i = (-2^1,1^i)\in\mathcal{W}$ for each $i=2,\dots,r$. If $r\geq 3$, then $\textbf{w}'_2,\textbf{w}'_3$ are adjacent without elements of $\mathcal{W}$ on $\overline{\textbf{w}'_2\textbf{w}'_3}$, so that $J(\textbf{d}+\textbf{w}'_2,\textbf{d}+\textbf{w}'_3)=0$, and thus $d_1=4$, a contradiction. Therefore $r=2$. Since $\textbf{w}_0=(-1,0)$ and $\textbf{w}'=(-2,1)$ are vertices of $\mathrm{conv}(\mathcal{W})$, we have $(0,-1)\notin\mathcal{W}$, and thus $d_2=1$. However, we have that $\textbf{c}_+$ is null and $\textbf{c}_+\in\{s\leq -2\}$, which is a contradiction since it can be explicitly checked that $\left\{\frac{x_1^2}{2} + x_2^2 =1\right\} \cap \{x_1+x_2\leq -2\}=\emptyset$.
	
	Suppose instead that $\textbf{w}_0$ is of type II, say $\textbf{w}_0=(1^1,-1^2,-1^3)$. Then (\ref{w0product}) becomes
	\begin{equation}\label{w0tII}
		1=\frac{w'_1}{d_1} - \frac{w'_2}{d_2} - \frac{w'_3}{d_3}.
	\end{equation}
	We must have $(-1^1,1^2,-1^3),(-1^1,-1^2,1^3)\in\mathcal{W}$, but they cannot be vertices, as then (\ref{w0product}) would become $1=-\frac{1}{d_1}\pm \frac{1}{d_2}\mp \frac{1}{d_3}$, which has no solutions. Hence, by Lemma 4.1(vii) of \cite{DancerWang2005}, we must have (in particular) the vertices $\textbf{w}_1=(-2^2,1^3)$ and $\textbf{w}_2=(1^2,-2^3)$. Both $J(\textbf{d}+\textbf{w}_1,\textbf{d}+\textbf{w}_0)=0$ and $J(\textbf{d}+\textbf{w}_2,\textbf{d}+\textbf{w}_0)=0$ cannot be simultaneously satisfied for any $\textbf{d}$, as can be checked explicitly using (\ref{w0tII}). Therefore $\textbf{w}_0$ cannot be of type II.
	
	Finally, suppose that $\textbf{w}_0$ is of type III, say $\textbf{w}_0=(1^1,-2^2)$. Then (\ref{w0product}) becomes
	\begin{equation}\label{w0tIII}
		1=\frac{w'_1}{d_1} -\frac{2w'_2}{d_2}.
	\end{equation}
	One checks that for any $i\geq 3$, $\textbf{w}'=(-1^1,-1^2,1^i)$ cannot be a vertex adjacent to $\textbf{w}_0$. Suppose that $(-1^1,-1^2,1^i)\in\mathcal{W}$. Then we must have that $(-2^2,1^i)$ is adjacent to $\textbf{w}_0$, and hence $d_2=4$. The condition (\ref{w0tIII}) then becomes
	\begin{equation*}
		1=\frac{w'_1}{d_1}-\frac{w'_2}{2}.
	\end{equation*}
	On the other side, we first suppose that $(1^1,-1^2,-1^i)$ is a vertex. Then $d_1=2$, in which case all vertices $\textbf{w}'$ of $\mathrm{conv}(\mathcal{W})$, with the possible exception of $\textbf{w}_0$, must satisfy $\frac{w'_1}{2} - \frac{w'_2}{2}=1$. However, this is impossible, as $(-1^2)\in\{\frac{w'_1}{2} - \frac{w'_2}{2}<1\}$ but $\textbf{w}_0\in\{\frac{w'_1}{2} - \frac{w'_2}{2}>1\}$. The only remaining possibility is that $(1^1,-2^i)$ is a vertex, so that $d_1=1$ by (\ref{w0tIII}). Similarly, $(-1^2)\in \{w'_1 - \frac{w'_2}{2}>1\}$ and $(1^1,-2^2)\in\{w'_1-\frac{w'_2}{2}<1\}$, which is a contradiction. Therefore there are no type II vectors in $\mathcal{W}$ with nonzero entries in the first two positions. 
	
	Consider the plane spanned by the first 3 coordinates (see Figure \ref{diagram}). There are no type II vectors, and furthermore (\ref{w0tIII}) implies that $(-2^1,1^3),(-2^1,1^2)\notin\mathcal{W}$. Since $(-1^1)$ cannot be a vertex, we have $d_1=1$. Let $\textbf{c}_\pm=\frac{1}{2}(\textbf{d}+\textbf{v}_\pm)$. Since $\textbf{v}_++\textbf{v}_-=2\textbf{w}_0=(2^1,-4^2)$, we must in particular have $v_{+,1}+v_{-,1}=2$, but also since both $\frac{1}{2}(\textbf{d}+\textbf{v}_\pm)$ are null and $d_1=1$, we must have $v_{\pm,1}^2<1$, which is a contradiction. Thus $(a,b)=(1,1)$ is impossible.
\end{proof}

\begin{lem}\label{10}
	If $(a,b)=(1,0)$, then the only set of weight vectors admitting a superpotential is the Bérard Bergery--Calabi ansatz.
\end{lem}
\begin{proof}
	First, we claim that there are exactly two null vertices, which are of the form $\frac{1}{2}(\textbf{d}\pm \textbf{v})$.
	Let $\frac{1}{2}(\textbf{d}-\textbf{v})$ be the (null) vertex in $\{s>-1\}$. Arguing using $\ker J(\textbf{d}-\textbf{v},\cdot)$, we see that there must be at least one null vertex in $\{s=-1\}$, which by the closest point argument must be $\frac{1}{2}(\textbf{d}+\textbf{v})$.
	
	Let $\frac{1}{2}(\textbf{d}+\textbf{w})$ be a (non-null) vertex adjacent to $\frac{1}{2}(\textbf{d}+\textbf{v})$. It is also adjacent to $\frac{1}{2}(\textbf{d}-\textbf{v})$, so that $J(\textbf{d}-\textbf{v},\textbf{d}+\textbf{w})=0$. This implies that $J(\textbf{d}+\textbf{v},\textbf{d}+\textbf{w})\neq 0$, so either $\frac{1}{2}(\textbf{d}+\textbf{w})$, $\frac{1}{2}(\textbf{d}+\textbf{v})$ do not satisfy the unique sum condition, or $\frac{1}{2}(\textbf{v}+\textbf{w})\in\mathcal{W}$.
	
	Suppose the unique sum condition is not satisfied. Then, we consider the line segment $\overline{\textbf{v}\textbf{w}}$ and define on it the points $\textbf{x}_v,\textbf{x}_w$ closest to but distinct from $\textbf{v},\textbf{w}$ respectively. The midpoint of $\textbf{v}$, $\textbf{x}_v$, the midpoint of $\textbf{x}_w,\textbf{w}$ and $\textbf{w}$ are all in $\mathcal{W}$. If $\textbf{x}_v\neq \textbf{x}_w$, then this is impossible since the distance between the two midpoints is more than twice the distance between one of them and $\textbf{w}$. Thus $\textbf{x}_v=\textbf{x}_w$. Consequently the only possibility, up to permutation of indices, is that $\textbf{w}=(-2^1,1^2)$, $\textbf{x}_v=\textbf{x}_w=(-1^2)$ and $\textbf{v}=(2^1,-3^2)$. Then the condition for $\frac{1}{2}(\textbf{d}\pm \textbf{v})$ to be null is $1=\frac{4}{d_1} + \frac{9}{d_2}$ and the condition $J(\textbf{d}-\textbf{v}, \textbf{d}+\textbf{w})=0$ is $1=\frac{4}{d_1} + \frac{3}{d_2}$.	The two conditions are evidently contradictory, so the unique sum condition is satisfied, and $\frac{1}{2}(\textbf{v}+\textbf{w})\in\mathcal{W}$.
	
	Since $\mathrm{conv}(\mathcal{W})$ has nonzero $(r-1)$-dimensional measure, there are at least $(r-1)$ vertices in $\{s=-1\}$ adjacent to $\frac{1}{2}(\textbf{d}+\textbf{v})$, which we denote by $\frac{1}{2}(\textbf{d}+\textbf{w}_1),\dots,\frac{1}{2}(\textbf{d}+\textbf{w}_{r-1})$. For each $i$, we would then have that $\textbf{w}'_i \equiv \frac{1}{2}(\textbf{v}+\textbf{w}_i)\in\mathcal{W}$. Thus
	\begin{equation*}
		2\textbf{w}'_1-\textbf{w}_1=\dots=2\textbf{w}'_{r-1}-\textbf{w}_{r-1}=\textbf{v},
	\end{equation*}
	and
	\begin{equation}\label{nullroof}
		J(\textbf{d}-\textbf{v},\textbf{d}+\textbf{w}_i)=0.
	\end{equation}
	Thus the $\{\frac{1}{2}(\textbf{d}+\textbf{w}_i)\}$ and the $\{\frac{1}{2}(\textbf{d}+\textbf{w}'_i)\}$ lie on parallel affine hyperplanes of $\{s=-1\}$.

	Suppose that one of the $\textbf{w}_i$'s is a type I vector, say $\textbf{w}_1=(-1^1)$. Then $J(\textbf{d}-\textbf{v},\textbf{d}+\textbf{w}_1)=0$, so $1=-\frac{-v_1}{d_1}$, $v_1=d_1$. Since $\textbf{v}$ is null, we have $d_1=1$, a contradiction since $(-1^1)\in\mathcal{W}$.
	
	Suppose that one of them is a type II vector, say $\textbf{w}_1=(1^1,-1^2,-1^3)$. Then (\ref{nullroof}) is
	\begin{equation*}
		-\frac{v_1}{d_1} + \frac{v_2}{d_2} + \frac{v_3}{d_3}=1.
	\end{equation*}
	The null condition for $\frac{1}{2}(\textbf{d}\pm\textbf{v})$ gives
	\begin{equation*}
		\frac{v_1^2}{d_1} + \frac{v_2^2}{d_2} + \frac{v_3^2}{d_3}\leq 1.
	\end{equation*}
	Note that the only way that these conditions can hold for integers $v_1,v_2,v_3$ is if they satisfy $|v_1|,|v_2|,|v_3|\leq 1$.
	Suppose that two of $v_1,v_2,v_3$ are zero, say $v_2$ and $v_3$. Then $-v_1=d_1\geq 2$ (Proposition 4.1(iv) of \cite{DancerWang2005}), so $v_1^2/d_1>1$, a contradiction.
	Suppose that one of $v_1,v_2,v_3$ is zero, say $v_3$. Then
	\begin{equation*}
		-\frac{v_1}{d_1} + \frac{v_2}{d_2} = 1, \qquad \frac{v_1^2}{d_1} + \frac{v_2^2}{d_2}\leq 1.
	\end{equation*}
	Consequently, $d_1=d_2=2$, $v_1=-1$, $v_2=1$, and the right-hand side inequality is an equality, so $\textbf{v}=(-1^1,1^1)$, contradicting $\sum_{i=1}^r v_i = -1$. Hence, all three of $v_1,v_2,v_3$ are equal to $\pm 1$. In fact, since $d_1,d_2,d_3\geq 2$, we must have $v_1=-1,v_2=1,v_3=1$, and the null condition inequality is an equality, implying that $\textbf{v}=(-1^1,1^2,1^3)$, contradicting $\sum_{i=1}^r v_i = -1$.
	
	Hence, by elimination, each $\textbf{w}_i$ is a type III vector.
	Suppose without loss of generality that $\textbf{w}_1=(1^1,-2^2)$. Then
	\begin{equation}\label{witIII}
		-\frac{v_1}{d_1} + \frac{2v_2}{d_2}=1, \qquad \frac{v_1^2}{d_1} + \frac{v_2^2}{d_2}\leq 1.
	\end{equation}
	We observe immediately that $v_1\leq 0$.

	Suppose for now that $r\geq 3$,  and consider the equality 
	\begin{equation*}
		\textbf{w}_i-\textbf{w}_j=2(\textbf{w}'_i-\textbf{w}'_j).
	\end{equation*}
	
	The only way in which the difference of two type III vectors can have only even entries is if the $(-1)$'s are in the same position. Since there must be $r-1$ such vectors, the only possibility up to permutation is $\textbf{w}_i = (1^1,-2^{i+1})$, $i=1,\dots,r-1$. All the $\textbf{w}_i$'s lie on the hyperplane $\{x_1=1\}\cap \{\sum_{i=1}^r x_i = -1\}$. Since $2\textbf{w}'_i-\textbf{w}_i=\textbf{v}$ for each $i$, we must have $(\textbf{w}'_1)_1=(\textbf{w}'_2)_1=\dots =(\textbf{w}'_{r-1})_1$.
	
	Lemma 4.5 of \cite{DancerWang2005} yields the following dichotomy.
	
	First, $d_1=1$. Then $v_1\in \{-1,0\}$. Since $v_1$ must differ from $-(\textbf{w}_i)_1=-1$ by an even integer, we have $v_1=-1$. Since $\frac{1}{2}(\textbf{d}+\textbf{v})$ is null, we have $\textbf{v}=(-1,0,\dots,0)$ and we have exactly Example 18 of \cite{BetancourtDancerWang2016}, on the Bérard Bergery--Calabi ansatz.
	
	Second, $d_1>1$ and $(1^1,-1^i,-1^j)\in\mathcal{W}$ for each $i,j\in\{2,\dots,r\}$, $i\neq j$. Also, we have $i\notin S_1$ for each $i=2,\dots,r$, using the notation of \cite{DancerWang2005}. Since by definition $i\notin S_0$, we must have $i\in S_{\geq 2}$ for each $i=2,\dots,r$. This implies, by Proposition 4.2 of \cite{DancerWang2005}, that $d_2=d_3=\dots=d_r=4$. We know from above that $v_1\leq 0$; in fact $v_1\leq -1$. Then, for each $k=2,\dots,r$,
	\begin{equation*}
		-\frac{v_1}{d_1}+\frac{v_k}{2}=1, \qquad \frac{v_1^2}{d_1}+\sum_{i=2}^r\frac{v_i^2}{4} = 1.
	\end{equation*}
	In particular, we must have $|v_1/d_1|\leq 1$, with equality if and only if $v_1=-d_1$ and $\textbf{v}=(-d_1,0,\dots,0)$. However, we have assumed that $d_1>1$, so equality is impossible. Therefore $|v_1/d_1|<1$ and $v_k/2>0$ for each $k$. The only possibility is $v_k=1$ for each $k$, with $v_1/d_1=-1/2$. The average of $v_1$ and $1$ must be a component of a vector in $\mathcal{W}$, so in fact $v_1\in \{-1,-3,-5\}$. The second equation implies that in fact $v_1=-1$, $d_1=2$, so that $\sum_{i=1}^r v_i \geq 0$, which is a contradiction.
	
	In the case $r=2$, there are only three possibilities for $\textbf{w}'_1$. First, if $\textbf{w}'_1=(0,-1)$, then $\textbf{v}=(1,0)$ and we again have the Bérard Bergery--Calabi ansatz. Otherwise, either $\textbf{w}_1'=(0,-1)$ and $\textbf{v}=(-3,2)$ or $\textbf{w}_1'=(-2,1)$ and $\textbf{v}=(-5,4)$; in both of these cases, the null condition for $\frac{1}{2}(\textbf{d}+\textbf{v})$ contradicts (\ref{witIII}). This concludes the classification for the case $(a,b)=(1,0)$.
\end{proof}

\begin{lem}\label{newsuperpotential}
	If $(a,b)=(2,1)$, then two sets of weight vectors $\mathcal{W}$ admit superpotentials, namely $\mathcal{W} = \{(-1,0),(0,-1)\}$ and $\mathcal{W}=\{(-1,0,0),(0,-1,0),(-2,0,1),(0,-2,1)\}$.
\end{lem}
\begin{proof}
	Let the two (adjacent) vertices in $\{s>-1\}$ be $\frac{1}{2}(\textbf{d}\pm \textbf{v})$ and let $\frac{1}{2}(\textbf{d}+\textbf{u})$ be the vertex in $\{s<-1\}$. Let $Q_
	\pm = \ker J(\textbf{d}\pm\textbf{v},\cdot)\cap \{s=-1\}$ and $Q_u = \ker J(\textbf{d}+\textbf{u},\cdot)\cap \{ s=-1\}$. Each is a hyperplane of $\{s=-1\}$, and $Q_+\cap Q_- = \emptyset$.
	
	For convenience, where there is no ambiguity, we will sometimes identify a vector $\frac{1}{2}(\textbf{d}+\textbf{x})$ with $\textbf{x}$ and $\frac{1}{2}(\textbf{d}+(\mathbb{R}^r,0))= P$ with $\mathbb{R}^r$. By considering $\ker J(\textbf{d}+\textbf{u},\cdot)$, we note that $\textbf{u}$ must be adjacent to at least one of $\pm \textbf{v}$. Suppose $\textbf{u}$ is adjacent to $\textbf{v}$, but not to $-\textbf{v}$. Then $\frac{1}{2}(\textbf{u}+\textbf{v})\equiv\textbf{w}_0\in\mathrm{conv}(\mathcal{W})$. Since $\mathrm{conv}(\mathcal{W})$ has nonzero $(r-1)$-dimensional measure, $\textbf{w}_0$ has at least $r-1$ neighbours $\textbf{w}_1,\dots,\textbf{w}_{r-1}$ such that $T\equiv \mathrm{conv}(\{\textbf{w}_1,\dots,\textbf{w}_{r-1}\})$ is an $(r-2)$-dimensional subset of $\{s=-1\}$. Since $-\textbf{v}$ is not adjacent to $\textbf{u}$, each $\textbf{w}_i$ is in fact a vertex of $\mathrm{conv}(\mathcal{C})$ and is outside the plane spanned by $\pm \textbf{v}$ and $\textbf{u}$. Therefore $\textbf{w}_i$ cannot be adjacent to $-\textbf{v}$, so we must have that $\textbf{w}_i$, $\textbf{v}$ are adjacent, $\textbf{w}_i\in Q_+$ and $T\subset Q_+$.
	
	Consider the vertex $-\textbf{v}$ of $\mathrm{conv}(\mathcal{C})$, which has at least $r$ neighbours. Since we assume it is not adjacent to $\textbf{u}$, at least $r-1$ of these neighbours are in $\{s=-1\}$. Denote them by $\textbf{x}_1,\dots, \textbf{x}_{r-1}$; their convex hull $U\subset Q_-$ is $(r-2)$-dimensional.	We also have that $T,U\subset Q_u$. Therefore $T\subset Q_+\cap Q_u$ and $U\subset Q_-\cap Q_u$, which is a contradiction since $Q_+\cap Q_-=\emptyset$ implies that at least one of $Q_+\cap Q_u$ and $Q_-\cap Q_u$ has zero $(r-2)$-dimensional measure.
	
	Thus $\textbf{u}$ is adjacent to both $\pm \textbf{v}$, and the closest point argument yields that $\frac{1}{2}(\textbf{u}+\textbf{v})\equiv \textbf{w}_+\in\mathcal{W}$ and $\frac{1}{2}(\textbf{u}-\textbf{v})\equiv \textbf{w}_-\in\mathcal{W}$. Suppose also that the triangle $\textbf{v},-\textbf{v},\textbf{u}$ is not a face of $\mathrm{conv}(\mathcal{W})$. Then we may define $\textbf{w}_1,\dots,\textbf{w}_{r-1}$ spanning an $(r-2)$-dimensional set $T\subset Q_+$ as in the previous case, since we would have $\textbf{w}_-\notin T$. Similarly, we may define $\textbf{x}_1,\dots,\textbf{x}_{r-1}$ spanning an $(r-2)$-dimensional set $U\subset Q_-$, since we would have $\textbf{w}_+\notin U$. By the same reason as above, this is a contradiction.
	
	The remaining case is where the triangle $\textbf{v},-\textbf{v},\textbf{u}$ is a face of $\mathrm{conv}(\mathcal{C})$, or equivalently $\overline{\textbf{w}_+\textbf{w}_-}$ is an edge of $\mathrm{conv}(\mathcal{W})$. We compute 
	\begin{equation}\label{wpm1}
		J(\textbf{d}+\textbf{w}_+,\textbf{d}+\textbf{w}_-)=\frac{1}{4} (J(\textbf{d}+\textbf{u},\textbf{d}-\textbf{v})+J(\textbf{d}+\textbf{u},\textbf{d}+\textbf{v})+J(\textbf{d}+\textbf{v}, \textbf{d}-\textbf{v}))=1.
	\end{equation}
	
	As a vertex of $\mathrm{conv}(\mathcal{W})$, $\textbf{w}_+$ has at least $r-1$ neighbours, out of which at least $r-2$ are not $\textbf{w}_-$ and therefore are also vertices of $\mathrm{conv}(\mathcal{C})$. Denote them by $\textbf{x}_1,\dots,\textbf{x}_{r-2}$. They are all outside the face $\textbf{v},-\textbf{v},\textbf{u}$, so they are all adjacent to $\textbf{v}$. Therefore we have $J(\textbf{d}+\textbf{u},\textbf{d}+\textbf{x}_i)=J(\textbf{d}+\textbf{v},\textbf{d}+\textbf{x}_i)=0$, which together with the relations $\textbf{u}=\textbf{w}_++\textbf{w}_-$ and $\textbf{v}=\textbf{w}_+-\textbf{w}_-$ gives $J(\textbf{d}+\textbf{w}_+,\textbf{d}+\textbf{x}_i)=0$ and $J(\textbf{d}+\textbf{w}_-,\textbf{d}+\textbf{x}_i)=1$.

	We consider the situation case by case, choosing each of $\textbf{w}_\pm$ to be a type I, II or III vector in turn.
	
	First, we claim that if $\textbf{w}_+$ is of type II, say $\textbf{w}_+=(1^1,-1^2,-1^3)$, then $\textbf{w}_-$ can only have nonzero entries in the first $3$ coordinates. This reduces the number of cases we will have to check. Suppose otherwise for sake of contradiction. Then $\textbf{w}_-$ lies outside of the plane $R$ spanned by the first three coordinates, so that each vertex of $\mathrm{conv}(\mathcal{W}\cap R)$ is a vertex of $\mathrm{conv}(\mathcal{C})$. First, notice that $\textbf{w}_+$ cannot be adjacent to either $(-1^1,-1^2,1^3)$ or $(-1^1,1^2,-1^3)$, since we would then have $-\frac{1}{d_1} \pm \frac{1}{d_2} \mp \frac{1}{d_3} = 1$, which is impossible. Also, it cannot be adjacent to either $(-2^2,1^3)$ or $(1^2,-2^3)$, since we would have
	\begin{equation*}
		\frac{2}{d_2} - \frac{1}{d_3} = 1 \quad \mathrm{or} \quad -\frac{1}{d_2} + \frac{2}{d_3} = 1,
	\end{equation*}
	which are impossible since $d_2,d_3\geq 2$. Therefore $(1^1,-2^2)$ and $(1^1,-2^3)$ are both adjacent to $\textbf{w}_+$, which contradicts the assumption that $\textbf{w}_-$ is a vertex of $\mathrm{conv}(\mathcal{W})$.
	
	If both $\textbf{w}_\pm$ are of type I, then without loss of generality $\textbf{w}_+=(-1^1)$ and $\textbf{w}_-=(-1^2)$. Then $\textbf{u}=(-1^1,-1^2)$ and $\textbf{v}=(-1^1,1^2)$, leading to the null condition $\frac{1}{d_1} + \frac{1}{d_2}=1$, implying that $d_1=d_2=2$. Suppose that $\textbf{w}_+$ has no neighbours in $\mathrm{conv}(\mathcal{W})$ other than $\textbf{w}_-$. Then we have exactly $\mathcal{W}=\{(-1^1),(-1^2)\}$ and the superpotential discussed in \cite[Example 16]{BetancourtDancerWang2016}. 
	
	Suppose that $\textbf{w}_+$ has a neighbour $\textbf{x}$ outside of the plane formed by $\textbf{v}, -\textbf{v}$ and $\textbf{u}$. Then $J(\textbf{d}+\textbf{w}_+,\textbf{d}+\textbf{x})=0$ becomes $1=-x_1/d_1$, and therefore $x_1=-2$ and $\textbf{x}=(-2^1,1^i)$ for some $i\geq 2$. In particular, there can only be one such out-of-plane neighbour, since if $i\neq j$, we would have the adjacent vertices $(-2^1,1^i)$ and $(-2^1,1^j)$, so that $d_1=4$, a contradiction. Therefore $r=3$ and $\textbf{x}=(-2^1,1^3)$. Thus $\textbf{w}_-$ would also have an adjacent vertex, which would have to be $\textbf{y}\equiv (-2^2,1^3)$. It is clear that $\textbf{x}$ and $\textbf{y}$ are adjacent and satisfy the unique sum condition. To see this, suppose otherwise, then $(-1,-1,1)\in\mathcal{W}$. However, the triangle $\textbf{u},\textbf{v},-\textbf{v}$, which is contained in the plane $\{x_3=0\}$, is a face of $\mathrm{conv}(\mathcal{C})$, and $(-1,-1,1)\in \{x_3>0\}$ while $(1,-1,-1),(-1,1,-1)\in\{x_3<0\}$. Therefore $J(\textbf{d}+\textbf{x},\textbf{d}+\textbf{y})=0$, implying that $d_3=1$.
	
	Thus, we have obtained a special case of the Bérard Bergery--Calabi ansatz discussed in \cite{BetancourtDancerWang2016}. We permute the indices $1\leftrightarrow 3$ to match the notation, so that the geometric data is $(d_1,d_2,d_3)=(1,2,2)$ and $\mathcal{W}=\{(0,-1,0),(0,0,-1),(1,-2,0),(1,0,-2)\}$. However, the superpotential vectors differ from those in Lemma \ref{10} and \cite[Example 18]{BetancourtDancerWang2016}. They are $\mathcal{C} = \frac{1}{2}(\textbf{d}+(\mathcal{B}, 0))$, where 
	\begin{equation*}
		\mathcal{B} =\{\textbf{u},\textbf{v},-\textbf{v},\textbf{x},\textbf{y}\}= \{(0,-1,-1),(0,1,-1),(0,-1,1),(1,0,-2),(1,-2,0)\}.
	\end{equation*}
	Returning to (\ref{superpotentialcases}), and using that $\frac{1}{2}(\textbf{u}+\textbf{v}) =(0,0,-1)$ and $\frac{1}{2}(\textbf{u}-\textbf{v}) =(0,-1,0)$, we get
	\begin{equation*}
		\frac{1}{2}f_\textbf{u}f_\textbf{v} = A_{(0,0,-1)},\qquad \frac{1}{2}f_\textbf{u}f_{-\textbf{v}} = A_{(0,-1,0)} \quad \mathrm{and}\quad f_\textbf{v}f_{-\textbf{v}}=E.
	\end{equation*}
	Note that $A_{(0,0,-1)},A_{(0,-1,0)}>0$. We solve the above to obtain, choosing the positive square root for all three and assuming that $E>0$,
	\begin{equation*}
		f_\textbf{v} = \sqrt{\frac{EA_{(0,0,-1)}}{A_{(0,-1,0)}}}, \qquad f_{-\textbf{v}} = \sqrt{\frac{EA_{(0,-1,0)}}{A_{(0,0,-1)}}}, \qquad f_\textbf{u}=\frac{2}{\sqrt{E}}\sqrt{A_{(0,0,-1)}A_{(0,-1,0)}}.
	\end{equation*}
	
	For the vertices $\textbf{x}=(1,0,-2)$ and $\textbf{y}=(1,-2,0)$, we get
	\begin{equation*}
		-\frac{1}{2}f_\textbf{x}^2=A_{(1,0,-2)}, \qquad -\frac{1}{2}f_\textbf{y}^2=A_{(1,-2,0)}.
	\end{equation*}
	
	The signs are consistent with real solutions, since $A_{(1,0,-2)},A_{(1,-2,0)}<0$. The final relation to check is not a unique sum, because $\frac{1}{2}(\textbf{v}+\textbf{y})=\frac{1}{2}(-\textbf{v}+\textbf{x})=(\frac{1}{2},-\frac{1}{2},-\frac{1}{2})$. We therefore have
	\begin{equation*}
		f_\textbf{v}f_\textbf{y} + f_{-\textbf{v}}f_\textbf{x}=0.
	\end{equation*}
	Thus, we must take opposite-sign square roots for $f_\textbf{x}$ and $f_\textbf{y}$, giving
	\begin{equation*}
		f_\textbf{x}=\pm\sqrt{-2A_{(1,0,-2)}}, \qquad f_\textbf{y}=\mp \sqrt{-2A_{(1,-2,0)}}.
	\end{equation*}
	The preceding equation then gives the requirement
	\begin{equation}\label{newexrequirement}
		\frac{A_{(0,0,-1)}^2}{A_{(1,0,-2)}}=\frac{A_{(0,-1,0)}^2}{A_{(1,-2,0)}}.
	\end{equation}
	If (\ref{newexrequirement}) is satisfied, then there is a superpotential. Note that this is exactly the same requirement as in \cite[Example 18]{BetancourtDancerWang2016}.
	
	If $\textbf{w}_+$ is of type I, say $\textbf{w}_+=(-1^1)$, and $\textbf{w}_-$ is of type II, then by (\ref{wpm1}), we must have that the nonzero components of $\textbf{w}_-$ are disjoint from those of $\textbf{w}_+$, which is impossible.
	
	If $\textbf{w}_+=(-1^1)$ and $\textbf{w}_-$ is of type III, then the components are disjoint by (\ref{wpm1}), so without loss of generality $\textbf{w}_-=(1^2,-2^3)$, in which case $\textbf{u}=(-1^1,1^2,-2^3)$ and $\textbf{v}=(-1^1,-1^2,2^3)$. The null condition for $\textbf{u},\pm\textbf{v}$ is then
	\begin{equation*}
		\frac{1}{d_1} + \frac{1}{d_2} + \frac{4}{d_3}=1.
	\end{equation*}
	Consider the plane $R$; it contains both $\textbf{w}_\pm$, and also $(-1^3)$. Since the segment $\overline{\textbf{w}_+\textbf{w}_-}$ is an edge of $\mathrm{conv}(\mathcal{W})$, we must have that $(-2^1,1^2),(-2^1,1^3), (-1^1,1^2,-1^3)\notin \mathcal{W}$, and thus also $(1^1,-1^2,-1^3),(-1^1,-1^2,1^3)\notin\mathcal{W}$.
	
	Each vertex $\textbf{x}$ in $R$ adjacent to $\textbf{w}_+=(-1^1)$ satisfies $J(\textbf{d}+\textbf{w}_+,\textbf{d}+\textbf{x})=0$ and thus $1=-x_1/d_1$, or $x_1<0$, which is not satisfied by any of the remaining vectors of $\mathcal{W}$ in $R$. Hence, if $\textbf{w}_+$ is of type I, $\textbf{w}_-$ cannot be of type III.
	
	Suppose $\textbf{w}_+$ is of type II, say $\textbf{w}_+=(1^1,-1^2,-1^3)$, and that $\textbf{w}_-$ is also of type II and thus lies in $R$. Say $\textbf{w}_-=(-1^1,1^2,-1^3)$, then $\textbf{u}=(-2^3)$ and $\textbf{v}=(2^1,-2^2)$. The null condition requires that $d_3=4$. The conditions $J(\textbf{d}+\textbf{w}_+,\textbf{d}+\textbf{w}_-)=0$ and $J(\textbf{d}\pm\textbf{v},\textbf{d}\pm \textbf{v})=0$ are equivalent and both reduce to $\frac{4}{d_1} + \frac{4}{d_2}=1$. Consider again the plane spanned by the first three coordinates.
	Since $\overline{\textbf{w}_+\textbf{w}_-}$ is an edge, we have $(1^1,-2^2),(1^1,-2^3)\notin\mathcal{W}$. Furthermore, $(1^1,-2^2,0)$, $(0,-2,1)$ and $(-1^1,-1^2,1^3)$ cannot be vertices adjacent to $\textbf{w}_+$, since they are not in $\ker J(\textbf{d}+\textbf{w}_+,\textbf{d}+\cdot)$. This is a contradiction, as can be seen from Figure \ref{diagram} and the fact that $(-1^1,-1^2,1^3)\in\mathcal{W}$.
	
	Suppose that $\textbf{w}_+=(1^1,-1^2,-1^3)$ is of type II and $\textbf{w}_-$ is of type III. By (\ref{wpm1}), we must have $\textbf{w}_-=(1^2,-2^3)$, up to permutation $2\leftrightarrow 3$. This leads to $d_3=2d_2$. Then $\textbf{u}=(1^1,-3^3)$ and $\textbf{v}=(1^1,-2^2,1^3)$. The null condition for $\textbf{u},\textbf{v}$ is then
	\begin{equation}\label{(2,1)tIItIIIshare2nullcondition}
		\frac{1}{d_1} + \frac{9}{d_3}=1.
	\end{equation}
	Consider the neighbours of $\textbf{w}_+$. If $(1^1,-2^2)$ is adjacent to $\textbf{w}_+$, then $1 = \frac{1}{d_1} + \frac{2}{d_2} = \frac{1}{d_1} + \frac{4}{d_3}$, contradicting (\ref{(2,1)tIItIIIshare2nullcondition}). We have that $(-2^2,1^3)$ and $(-1^1,-1^2,1^3)$ cannot be adjacent to $\textbf{w}_+$ for the same reasons as above. Since $(-1^1,-1^2,1^3)\in\mathcal{W}$, we have a contradiction, as made clear by Figure \ref{diagram}.

	The last remaining case is when $\textbf{w}_\pm$ are both of type III. By (\ref{wpm1}), they cannot share nonzero entries, and thus we may take without loss of generality $\textbf{w}_+=(1^1,-2^2)$ and $\textbf{w}_-=(1^3,-2^4)$. Then $\textbf{u}=(1^1,-2^2,1^3,-2^4)$ and $\textbf{v}=(1^1,-2^2,-1^3,2^4)$. The null condition for $\textbf{u},\textbf{v}$ is
	\begin{equation}\label{(2,1)tIIItIIInull}
		\frac{1}{d_1} + \frac{4}{d_2} + \frac{1}{d_3} + \frac{4}{d_4}=1.
	\end{equation}
	Consider the neighbours of $\textbf{w}_+$ in the plane $R$. Then $\textbf{w}_+$ cannot be adjacent to the following vertices:
	\begin{itemize}
		\item $(-2^2,1^3)$, because in that case $d_2=4$, contradicting (\ref{(2,1)tIIItIIInull}).
		\item $(-1^1,-1^2,1^3)$, as that would give $-\frac{1}{d_1} + \frac{2}{d_2}=1$, which is impossible as $d_2\geq 2$.
		\item $(-2^1,1^3)$, as that would mean $1=-\frac{2}{d_1}$.
		\item $(-2^1, 1^2)$, since that gives $1=-\frac{2}{d_1} - \frac{2}{d_2}$.
		\item $(-1^1)$, as it gives $1=-\frac{1}{d_1}$.
	\end{itemize}
	Since $(-1^2)\in\mathcal{W}$ by Lemma 4.1(ii) of \cite{DancerWang2005}, and we have eliminated all other possibilities as can be seen in Figure \ref{diagram}, $(-1^2)$ must be adjacent to $\textbf{w}_+$, and therefore $1=\frac{2}{d_2}$ and $d_2=2$, contradicting (\ref{(2,1)tIIItIIInull}) and completing the proof of the lemma.
\end{proof}

We may summarise Lemmas \ref{20}, \ref{11}, \ref{10} and \ref{newsuperpotential} in the following statement. Superpotentials (\ref{sup1})-(\ref{sup4}) had already been found in \cite{BetancourtDancerWang2016}.

\begin{thm}\label{classification}
	Suppose that $K$ is connected. Then, under the assumption that there is no coordinate position in which all weight vectors vanish, the only configurations $\mathcal{W}$ admitting a superpotential of the form (\ref{superansatz}) are, up to permutation of the irreducible summands,
	\begin{enumerate}
		\item $\mathcal{W}=\emptyset$ with $d_1=1$. This is the 2-dimensional Bryant soliton, discussed in \cite[Example 15]{BetancourtDancerWang2016}. There is one superpotential, with
		$\mathcal{C}=\frac{1}{2}(\textbf{d}+\{(1,0),(-1,0)\})$.\label{sup1}
		\item $\mathcal{W}=\{(-1)\}$ with $d_1=4$. This is the 5-dimensional Bryant soliton, discussed in \cite[Example 15]{BetancourtDancerWang2016}. There is one superpotential, with $\mathcal{C} = \frac{1}{2}(\textbf{d}+\{(0,0),(-2,0)\})$.
		\item $\mathcal{W}=\{(-1,0),(0,-1)\}$ with $d_1=d_2=2$. This is the warped product of two Einstein 2-manifolds with positive Einstein constant, discussed in \cite[Example 16]{BetancourtDancerWang2016}. There is one superpotential, with $\mathcal{C}=\frac{1}{2}(\textbf{d}+\{(-1,-1,0),(-1,1,0),(1,-1,0)\})$.
		\item For each $r\geq 2$, $\mathcal{W}=\{(-1^2),\dots,(-1^r),(1^1,-2^2),\dots,(1^1,-2^r)\}$, with $d_1=1$ and $d_i=2m_i$ for positive integers $m_i$, $i=2,\dots,r$. This is the Bérard Bergery--Calabi ansatz, discussed in \cite[Example 18]{BetancourtDancerWang2016}. We have a superpotential with $\mathcal{C} =\frac{1}{2}(\textbf{d}+\{(-1)^1,(1^1,-2^2),\dots,(1^1,-2^r),(1^1)\})$.\label{sup4}
		\item In the special case $r=3$, $d_2=d_3=2$ of (4), there is an additional superpotential, with $\mathcal{C}=\frac{1}{2}(\textbf{d}+\{(0,-1,-1,0),(0,1,-1,0),(0,-1,1,0),(1,0,-2,0),(1,-2,0,0)\})$.\label{newsup}
	\end{enumerate}
\end{thm}

\medskip

\textit{Remark.} As in the Ricci-flat case, the only superpotential for $r>3$ is for the Bérard Bergery--Calabi ansatz. For $r=2,3$, only cases (2) and (3) of Theorem 1.14 of \cite{DancerWang2011b} are allowed, with further restrictions on the dimension.

\section{Superpotentials in the non-steady case}\label{nonsteadysection}

In this section, we discuss the application of the methods of the previous section to expanding or shrinking solitons. We show that even if we allow polynomial coefficients, no superpotential exists in the non-steady case except in the simplest case $n=1$. We also conclude from the proof that in the steady case it suffices to consider superpotentials with constant, rather than polynomial, coefficients.

We compute from the expression for the Hamiltonian $\mathcal{H}$ that the superpotential condition in the general case is
\begin{equation*}
	J(\nabla f, \nabla f) = e^{\mathbf{d}\cdot\mathbf{q}}\left( E - \lambda(n+1) + \lambda u + \sum_{w\in \mathcal{W}} A_w e^{\mathbf{w}\cdot \mathbf{q}}\right).
\end{equation*}
We note immediately that there are no superpotentials of the form (\ref{superansatz}) for expanding or shrinking solitons.

\subsection{Superpotentials with polynomial coefficients}
In view of the above, it is worth considering superpotentials with polynomial coefficients $f_c(\textbf{q})$. Then, the main relation (\ref{superpotentialcases}) is instead
\begin{multline*}
	\sum_{\textbf{a}+\textbf{c}=\textbf{b}} J(\textbf{a},\textbf{c})f_af_c + J(\nabla f_a, \textbf{c})f_c + J(\textbf{a},\nabla f_c) f_a + J(\nabla f_a, \nabla f_c) \\
	=
	\begin{cases*}
		A_w & if $\mathbf{b} = \mathbf{d}+\mathbf{w}$  for some $\mathbf{w}\in\mathcal{W}$\\
		E - \lambda(n+1) + \lambda u & if $\mathbf{b}=\mathbf{d}$\\
		0 & otherwise.
	\end{cases*}
\end{multline*}
Considering the top degree part $J(\textbf{a},\textbf{c})f_af_c$, we obtain immediately that Lemma \ref{lemma12} remains true. In particular, if $\textbf{c}$ is a non-null vertex, then $f_c$ is a constant polynomial. Note also that Proposition \ref{prop13} remains true.

If $\textbf{c}$ is a null vertex, then first note that $\textbf{c}\neq \frac{\textbf{d}}{2}$. Thus we must have that
\begin{equation*}
	2J(\nabla f_c, \textbf{c})f_c + J(\nabla f_c, \nabla f_c)= \begin{cases*} A_w & if $2\textbf{c}\in \textbf{d}+\mathcal{W}$\\
		0 & otherwise.
	\end{cases*}
\end{equation*}
The left-hand side must be a constant, so $J(\nabla f_c,\textbf{c})=0$, which also results in $J(\nabla f_c, \nabla f_c)=0$ and thus $2\textbf{c}\notin \textbf{d}+\mathcal{W}$.

Also, we can verify by considering the top degree terms that Lemma \ref{edgeproportional} remains true, and so do Corollary \ref{atmost2} and Lemmas \ref{bothsidesP}, \ref{CinP}, \ref{notd/2} and \ref{abrange}.

The following considerations rule out essentially every possibility of using polynomial coefficients to find a superpotential for non-steady solitons.

\begin{lem}\label{noa=2}
	For any $\lambda\in\mathbb{R}$ and $n>1$, there are no superpotentials with polynomial coefficients which have two adjacent null vertices $\frac{1}{2}(\textbf{d}+\textbf{a})$ and $\frac{1}{2}(\textbf{d}-\textbf{a})$ for some $\textbf{a}\in \mathbb{R}^r$.
\end{lem}
\begin{proof}
	Suppose for sake of contradiction that there are indeed two such vertices. Using the closest point argument, we see that they satisfy the unique sum condition. Thus
	\begin{multline*}
		J\left(\frac{1}{2}(\textbf{d}+\textbf{a}),\frac{1}{2}(\textbf{d}-\textbf{a})\right)f_af_{-a} + J\left(\nabla f_a, \frac{1}{2}(\textbf{d}-\textbf{a})\right)f_{-a} + J\left(\frac{1}{2}(\textbf{d}+\textbf{a}), \nabla f_{-a}\right) f_a \\
		+ J(\nabla f_a, \nabla f_{-a}) = E - \lambda (n+1) + \lambda u.
	\end{multline*}
	Since the highest order term is the first one, which does not vanish, and the right-hand side is a polynomial of degree at most one, at most one of $f_a$ and $f_{-a}$ is non-constant, say $f_a$. Then $f_a=Au+B$ for some constants $A, B$, and
	\begin{equation*}
		J\left(\frac{1}{2}(\textbf{d}+\textbf{a}),\frac{1}{2}(\textbf{d}-\textbf{a})\right)f_af_{-a} + J\left(\nabla f_a, \frac{1}{2}(\textbf{d}-\textbf{a})\right)f_{-a} = E - \lambda (n+1) + \lambda u.
	\end{equation*}
	The superpotential condition for $\frac{1}{2}(\textbf{d}+\textbf{a})\neq \frac{1}{2}\textbf{d}$ gives
	\begin{equation*}
		2J\left(\nabla f_a, \frac{1}{2}(\textbf{d}+\textbf{a})\right)f_a+ J(\nabla f_a,\nabla f_a)=const.
	\end{equation*}
	By degree considerations, we must have $J(\nabla f_a, \frac{1}{2}(\textbf{d}+\textbf{a}))=0$, and therefore $J(\nabla f_a, \nabla f_a)=-\frac{(n-1)}{4}A^2=0$ using the formula (\ref{J}) for $J$ and the linearity of $J$. Thus $f_a$ is constant unless $n=1$ (and $\textbf{a}\notin\mathcal{W}$).
\end{proof}
The above lemma rules out polynomial-coefficient superpotentials for the $a=2$ cases of Lemma \ref{abrange}. Furthermore, we note that in the case $(a,b)=(1,1)$, the beginning of the proof of Lemma \ref{11} remains valid, so the existence of adjacent vertices $\textbf{c}_\pm$ contradicts Lemma \ref{noa=2}. Finally, for $(a,b)=(1,0)$, note that by the beginning of the proof of Lemma \ref{10}, there are two adjacent null vertices $\frac{1}{2}(\textbf{d}\pm\textbf{v})$, which contradicts Lemma \ref{noa=2} when $n>1$. We have therefore concluded that
\begin{pro}
	Suppose that $K$ is connected and that $n>1$. Then, under the assumption that there is no coordinate position in which all weight vectors vanish, there is no configuration $\mathcal{W}$ which admits a superpotential with nonconstant polynomial coefficients for any $\lambda\in\mathbb{R}$.
\end{pro}
\subsubsection{The two-dimensional case}
In the $n=1$ case, the setup is simple, as the principal orbit has no scalar curvature, so $\tilde{\mathcal{W}}=\{(0,0)\}$, $\textbf{d}=(1,-2)$ and
\begin{equation*}
	J\left((p,\phi),(p',\phi')\right)=-\left(pp' + \frac{p'\phi}{2} + \frac{p\phi'}{2}\right).
\end{equation*}It can easily be verified that there is only one nontrivial arrangement of superpotential vectors \cite{BetancourtDancerWang2016}, namely $\mathcal{C}=\{\textbf{c}_1,\textbf{c}_2\}$, where
\begin{equation*}
	\textbf{c}_1=\frac{1}{2}(\textbf{d}+(1,0)) = (1,-1) \qquad \text{and} \qquad \textbf{c}_2=\frac{1}{2}(\textbf{d}+(-1,0))=(0,-1)
\end{equation*}
are both null. The superpotential condition at their sum is
\begin{equation}\label{n=1condition}
	2\left(J(\textbf{c}_1,\textbf{c}_2) f_1f_2 + J(\nabla f_1,\textbf{c}_2)f_2 + J(\textbf{c}_1,\nabla f_2)f_1 + J(\nabla f_1,\nabla f_2)\right) = E - \lambda(n+1)  +\lambda u.
\end{equation}
Since $J(\textbf{c}_1,\textbf{c}_2) = \frac{1}{2}\neq 0$, from degree considerations we obtain that $f_j = A_j u + B_j$ for some constants $A_j,B_j$, $j=1,2$. The superpotential condition at $\textbf{c}_j$ is then
\begin{equation*}
	2J(\nabla f_j, \textbf{c}_j) + J(\nabla f_j,\nabla f_j)=0.
\end{equation*}
Since $J(\nabla f_j,\nabla f_j)=0$, we have $J(\nabla f_1, \textbf{c}_1)=J(\nabla f_2,\textbf{c}_2)=0$ and consequently $A_1=0$. Noting that $J(\textbf{c}_1,\nabla f_2) =-A_2/2$, the condition (\ref{n=1condition}) becomes
\begin{equation*}
	2\left(\frac{1}{2}B_1(A_2u + B_2) - \frac{A_2B_2}{2}\right) = E-\lambda(n+1) + \lambda u.
\end{equation*}
Separating by degree, we get $A_2B_1=\lambda$ and $E-\lambda(n+1) = B_1(B_2-A_2)$. Fixing $B_1\equiv a$, then $A_2=\frac{\lambda}{a}$ and $B_2 = \frac{1}{a}(E-\lambda n)$. There is therefore the superpotential
\begin{equation*}
	f=f_1 e^{\textbf{c}_1\cdot \textbf{q}} + f_2 e^{\textbf{c}_2\cdot\textbf{q}} = ae^{q-u} + \frac{1}{a}\left(\lambda u + (E-\lambda n)\right) e^{-u}.
\end{equation*}
Note that in the $\lambda=0$ (i.e. steady) case, we recover the superpotential given in \cite[Example 15]{BetancourtDancerWang2016}.

The associated first-order subsystem, after substituting back $h=e^{\frac{q}{2}}$, is given by
\begin{align*}
	\dot{h} &= -\frac{a}{2}h^2 + \frac{1}{2a}(\lambda u + E - 2\lambda)\\
	\dot{u} &= -ah.
\end{align*}
Substituting gives
\begin{equation*}
	\ddot{h} + a\dot{h}h + \frac{\lambda}{2}h=0.
\end{equation*}
In order for the soliton to close up at a singular orbit, one requires that $h(0)=0$. We can choose the boundary condition $u(0)=0$, which implies that $\dot{h}(0)=\frac{1}{2a}(E-2\lambda)$.
	
\section{First order subsystem and explicit solution}\label{examplesection}

In this section, we take the superpotential of Theorem \ref{classification}(\ref{newsup}) and study the corresponding first-order subsystem of the Hamiltonian system. We find an explicit solution, which satisfies suitable boundary conditions so that it is a smooth steady soliton metric with hypersurfaces which are $S^1$-bundles over $S^2\times S^2$, collapsing to $S^2$ at the singular orbit.

For a superpotential $f$, the ODE subsystem takes the form \cite{BetancourtDancerWang2016}
\begin{equation}\label{firstordersubsystem}
	\dot{\textbf{q}}=2\textbf{v}^{-1}J\nabla f,
\end{equation}
where $\textbf{v}$ is the extended relative volume $\textbf{v}=\exp (\frac{1}{2}\textbf{d}\cdot \textbf{q})$.

\subsection{Singular initial value problem}

We have the geometric data $\textbf{d}=(1,2,2,-2)$ and $\mathcal{W}=\{(0,-1,0),(0,0,-1),(1,-2,0),(1,0,-2)\}$. To simplify the presentation, we will translate the coordinates $q_2$ and $q_3$ so that $A_{(0,0,-1)}=A_{(0,-1,0)}=\lambda$. Then the condition (\ref{newexrequirement}) becomes $A_{(1,0,-2)}=A_{(1,-2,0)}\equiv A$, and the superpotential is
\begin{multline*}
	f=\sqrt{E}\, e^{\frac{q_1}{2} + \frac{3q_2}{2} + \frac{q_3}{2} - u} + \sqrt{E}\, e^{\frac{q_1}{2} + \frac{q_2}{2} + \frac{3q_3}{2} - u} + \frac{2\lambda}{\sqrt{E}}\, e^{\frac{q_1}{2} + \frac{q_2}{2} + \frac{q_3}{2} -u}\\
	+\sqrt{-2A}\, e^{q_1+q_2-u} -\sqrt{-2A}\, e^{q_1 + q_3 -u}.
\end{multline*}
By direct computation, we see that the first-order system (\ref{firstordersubsystem}) takes the form
\begin{equation}\label{newODE}
	\begin{split}
		\dot{q}_1=&\, \sqrt{-2A} e^{\frac{q_1}{2} - q_2} - \sqrt{-2A} e^{\frac{q_1}{2} - q_3}\\
		\dot{q}_2=&\, -\frac{\sqrt{E}}{2} e^{\frac{q_2}{2} - \frac{q_3}{2}} +\frac{\sqrt{E}}{2} e^{-\frac{q_2}{2} + \frac{q_3}{2}} + \frac{\lambda}{\sqrt{E}} e^{-\frac{q_2}{2} - \frac{q_3}{2}} - \sqrt{-2A} e^{\frac{q_1}{2}-q_2}\\
		\dot{q}_3=&\,\frac{\sqrt{E}}{2} e^{\frac{q_2}{2} - \frac{q_3}{2}} -\frac{\sqrt{E}}{2} e^{-\frac{q_2}{2} + \frac{q_3}{2}} + \frac{\lambda}{\sqrt{E}} e^{-\frac{q_2}{2} - \frac{q_3}{2}} + \sqrt{-2A} e^{\frac{q_1}{2}-q_3}\\
		\dot{u}=&\,-\frac{\sqrt{E}}{2} e^{\frac{q_2}{2} - \frac{q_3}{2}} - \frac{\sqrt{E}}{2} e^{-\frac{q_2}{2} + \frac{q_3}{2}} + \frac{\lambda}{\sqrt{E}} e^{-\frac{q_2}{2} - \frac{q_3}{2}}.
	\end{split}
\end{equation}
Now, we will perform the change of variables
\begin{equation*}
	\alpha = -2A e^{q_1}, \qquad \beta_1 = e^{q_2}, \qquad \beta_2 = e^{q_3}.
\end{equation*}
The first three equations of (\ref{newODE}), which form the ODE system for $\alpha,\beta_1,\beta_2$, are then
\begin{equation*}
	\begin{split}
		\frac{\dot\alpha}{\alpha} &= \sqrt{\alpha} \left(\frac{1}{\beta_1} - \frac{1}{\beta_2}\right)\\
		\frac{\dot\beta_1}{\beta_1} &= \frac{\sqrt{E}}{2}\left(\frac{\beta_2-\beta_1}{\sqrt{\beta_1\beta_2}}\right) + \frac{\lambda}{\sqrt{E}} \frac{1}{\sqrt{\beta_1\beta_2}} - \frac{\sqrt{\alpha}}{\beta_1}\\
		\frac{\dot\beta_2}{\beta_2} &= -\frac{\sqrt{E}}{2}\left(\frac{\beta_2-\beta_1}{\sqrt{\beta_1\beta_2}}\right) + \frac{\lambda}{\sqrt{E}} \frac{1}{\sqrt{\beta_1\beta_2}} + \frac{\sqrt{\alpha}}{\beta_2}.
	\end{split}
\end{equation*}
Adding the three above equations results in
\begin{equation}\label{newcoordinateODE}
	\frac{\frac{d}{dt}(\alpha\beta_1\beta_2)}{\alpha\beta_1\beta_2} = \frac{2\lambda}{\sqrt{E}} \frac{1}{\sqrt{\beta_1\beta_2}}.
\end{equation}
Similarly to \cite{DancerWang2011}, we define the new radial coordinate $ds = \sqrt{\alpha}\,dt$. Then (\ref{newcoordinateODE}) becomes
\begin{equation*}
	\frac{d}{ds}(\alpha\beta_1\beta_2)=\sqrt{\alpha\beta_1\beta_2} \frac{2\lambda}{\sqrt{E}}.
\end{equation*}
We take the singular orbit of the soliton, where at least one of $\alpha$, $\beta_1$ or $\beta_2$ vanishes, to be at $s=0$. Therefore
\begin{equation*}
	\sqrt{\alpha} = \frac{s}{\sqrt{\beta_1\beta_2}} \frac{\lambda}{\sqrt{E}} \qquad \mathrm{and} \qquad \frac{d}{dt} = \frac{\lambda}{\sqrt{E}} \frac{s}{\sqrt{\beta_1\beta_2}} \frac{d}{ds}.
\end{equation*}
The equations for $\beta_1$ and $\beta_2$ become
\begin{equation}\label{foscolohaskinsODE}
	\begin{split}
		\beta_1' &= \frac{1}{s}\left(\frac{E}{2\lambda}\beta_1(\beta_2-\beta_1) + \beta_1\right)-1\\
		\beta_2' &= \frac{1}{s}\left(-\frac{E}{2\lambda} \beta_2(\beta_2-\beta_1) + \beta_2\right) +1,
	\end{split}
\end{equation}
where $'$ denotes $\frac{d}{ds}$. This is a singular initial value problem for which we will apply the criteria in \cite{FoscoloHaskins2017} to find solutions. Let $\beta_1^0=\beta_1(s=0)$ and $\beta_2^0=\beta_2(s=0)$. Then condition (i) of \cite[Theorem 4.7]{FoscoloHaskins2017} is
\begin{equation}\label{foscolohaskins(i)}
	\begin{split}
		\frac{E}{2\lambda}\beta^0_1(\beta^0_2-\beta^0_1) + \beta^0_1 &=0\\
		-\frac{E}{2\lambda} \beta^0_2(\beta^0_2-\beta^0_1) + \beta^0_2 &=0.
	\end{split}
\end{equation}
The system (\ref{foscolohaskins(i)}) has the solution $\beta^0_1=\beta^0_2=0$, which does not satisfy condition (ii) of \cite[Theorem 4.7]{FoscoloHaskins2017}. The only other solutions are $\beta^0_1=\frac{2\lambda}{E}$, $\beta^0_2=0$ and $\beta^0_1=0$, $\beta^0_2=\frac{2\lambda}{E}$.

We will look for solutions $\beta_1(s),\beta_2(s)$ to (\ref{foscolohaskinsODE}) which are linear in $s$. By order considerations, we must have that
\begin{equation*}
	\beta_1(s)=-qs+\beta^0_1, \qquad \beta_2(s)=-qs+\beta^0_2
\end{equation*}
for some $q$. Plugging this into (\ref{foscolohaskinsODE}), we obtain the solutions
\begin{equation*}
	\begin{cases}
		\beta_1=-s+\frac{2\lambda}{E}\\
		\beta_2=-s
	\end{cases}\qquad\mathrm{and}\qquad
	\begin{cases}
		\beta_1=s\\
		\beta_2=s+\frac{2\lambda}{E}
	\end{cases}.
\end{equation*}
We have chosen various signs so that the initial value problem is solved for $s\geq 0$; therefore we may reject the former solution and keep only the latter. Hence,
\begin{equation*}
	\sqrt{\alpha}=\frac{\lambda}{\sqrt{E}}\sqrt{\frac{s}{s+\frac{2\lambda}{E}}}.
\end{equation*}
Taking $s=0$ when $t=0$, we integrate $ds=\sqrt{\alpha}\,dt$ to get
\begin{equation*}
	t=\frac{\sqrt{E}}{\lambda}\int_0^s \sqrt{1+\frac{2\lambda}{E\sigma}}\,d\sigma = \sqrt{\frac{2s}{\lambda}}\sqrt{1+\frac{Es}{2\lambda}} + \frac{2}{\sqrt{E}}\arccoth\sqrt{1+\frac{2\lambda}{Es}}.
\end{equation*}
Finally, the last equation of (\ref{newODE}) simplifies to
\begin{equation*}
	su'=-\frac{E}{2\lambda}(\beta_1+\beta_2)+1.
\end{equation*}
With the above solution $\beta_1,\beta_2$, we get $u=-Es/\lambda$, with $u(0)$ arbitrarily chosen to be zero.

\subsection{Boundary conditions and complete soliton}
It remains to check that suitable boundary conditions are satisfied for the soliton metric to extend smoothly to the singular orbit. This has already been analysed for a larger class of metrics in \cite[\S 4]{DancerWang2011}. We will follow their notation and define $f(t)^2=e^{q_1}=\frac{\alpha}{-2A}$, $g_1(t)^2=\beta_1(t)=e^{q_2}$, $g_2(t)^2=\beta_2(t)=e^{q_3}$. We will also normalise the $S^2$ factors in the same way as in \cite{BetancourtDancerWang2016,DancerWang2011}, that is, $\lambda=4$. We have that $f(0)=g_1(0)=0$, so that the singular orbit is a copy of the second factor.

It suffices then to verify up to order 2 \cite[Lemma 3.2]{DancerWang2011} using the formulae
\begin{equation*}
	f(t)=\frac{4}{\sqrt{-2AE}}\left(1+\frac{8}{Es(t)}\right)^{-\frac{1}{2}}, \qquad g_1(t)=\sqrt{s(t)}, \qquad g_2(t)=\sqrt{s(t)+\frac{8}{E}}
\end{equation*}
that $f(t)$ is smooth and odd with $\dot{f}(0)=1$, $g_1(t)$ is smooth and odd with $\dot{g}_1(0)^2=1/2$, and $g_2(t)$ is smooth and even. 
All the above conditions hold, except that $\dot{f}(0)=\frac{1}{\sqrt{-2A}}$, so smoothness requires that $A=-1/2$. The soliton potential $u(t)=-Es(t)/4$ is also smooth and even, and thus extends smoothly to the singular orbit.

To summarise, let $V_i$, $i=1,2$ be two copies of $S^2$, with metrics $h_i=\frac{1}{\sqrt{2}}g_{S^2}$, where $g_{S^2}$ is the standard metric on the unit sphere, and let $\pi_i: V_1\times V_2\rightarrow V_i$ be the projection onto the $i$-th factor. Let $a_i\in H^2(V_i, \mathbb{Z})$ be indivisible cohomology classes. Finally, let $P_q$ denote the principal $U(1)$-bundle over $V_1\times V_2$ with Euler class $b_1\pi_1^*a_1 + b_2\pi_2^*a_2$. Since the boundary conditions require $A=-1/2=-d_i b_i^2/4 = -b_i^2/2$ \cite{BetancourtDancerWang2016}, we may take $b_1=b_2=-1$, so that the Euler class of $P_q$ is $-\pi_1^*a_1 -\pi_2^*a_2$ (which is the same as required by \cite[Theorem 4.20(i)]{DancerWang2011}). Let $\theta$ be the principal $U(1)$-connection on $P_q$ with curvature $\Omega = -\pi_1^*\eta_1 - \pi_2^*\eta_2$, where $\eta_i$ is the Kähler form of $h_i$.

\begin{pro}\label{explicitsoliton}
	For each $E>0$, consider the metric $\hat g = dt^2 + g_t$ on $(0,\infty)\times P_q$, where
	\begin{equation*}
		g_t=f(t)^2 \theta\otimes\theta + g_1(t)^2 \pi_1^*h_1 + g_2(t)^2 \pi_2^*h_2
	\end{equation*}
	is the metric on $P_q$ given by
	\begin{equation*}
		f(t)=\frac{4}{\sqrt{2E}}\left(1+\frac{8}{Es(t)}\right)^{-1/2},\qquad g_1(t)=\sqrt{s(t)},\qquad g_2=\sqrt{s(t)+\frac{8}{E}},
	\end{equation*}
	\begin{equation*}
		t=\sqrt{\frac{s}{2}} \sqrt{1+\frac{Es}{8}} + \frac{2}{\sqrt{E}}\arccoth\sqrt{1+\frac{8}{Es}}.
	\end{equation*}
	With the soliton potential $u(t)=-Es(t)/4$, $\hat g$ extends to a smooth steady complete gradient Ricci soliton on the compactification of $(0,\infty)\times P_q$ by adding $S^2$ at $t=0$.
\end{pro}

\section{Generalised first integrals}\label{firstintegralsection}

One technique to simplify the problem is to find a certain conserved quantity called a \textit{generalised first integral}. In \cite{BetancourtDancerWang2016}, a number of examples of generalised first integrals are found using various ansätze and heuristics. In this section, we will determine the limitations of some of these methods.

A generalised first integral for the Ricci soliton equations is defined as a function $F: M\rightarrow \mathbb{R}$ satisfying
\begin{equation}\label{GFI}
	\{F,\mathcal{H}\} = \Phi \mathcal{H},
\end{equation}
for some function $\Phi$, where $\{\cdot,\cdot\}$ is the Poisson bracket and $\mathcal{H}$ is the Hamiltonian of the system. Since we require that $\mathcal{H}=0$ for solutions to the soliton equations, we have that $F$ is conserved. We look for solutions to (\ref{GFI}) of the form
\begin{equation*}
	F=\sum_\mathbf{b} F_\mathbf{b} e^{\mathbf{b}\cdot\mathbf{q}}, \qquad \Phi = \sum_{\mathbf{b}} \Phi_\mathbf{b} e^{\mathbf{b}\cdot \mathbf{q}},
\end{equation*}
where, for each $\mathbf{b}$, $F_\mathbf{b}$ and $\Phi_\mathbf{b}$ are polynomials in $\mathbf{p}$. Letting $\psi = \Phi - \frac{1}{2}\mathbf{d}\cdot\nabla_\mathbf{p}F$, we obtain that (\ref{GFI}) is equivalent to the recurrence relation shown in \cite[(4.2)]{BetancourtDancerWang2016}.

The strategy used in \cite{BetancourtDancerWang2016} to find generalised first integrals for the Bryant soliton is to factor the quadratic form
\begin{equation}\label{J}
	J(\textbf{p},\textbf{p})=-\sum_{i=1}^r \frac{p_i^2}{d_i} - \phi \sum_{i=1}^r p_i - \frac{(n-1)}{4}\phi^2
\end{equation}
as
\begin{equation}\label{factorisation}
	J=(\mathbf{c}\cdot \nabla J)\theta,
\end{equation}
where $\mathbf{c}$ is a fixed vector and $\theta$ is some function of $\textbf{p}$. We show that no factorisation exists when $\mathfrak{p}$ is not $\mathrm{Ad}(K)$-irreducible, i.e. when $r\geq 2$. 
\begin{pro}\label{r>1}
	If $r\geq 2$, then there is no factorisation of the quadratic form (\ref{J}) of the form (\ref{factorisation}).
\end{pro}
\begin{proof}
	We may compute
	\begin{equation*}
		\nabla J = -\left(\frac{2p_1}{d_1}+\phi,\: \dots\:,\: \frac{2p_r}{d_r}+\phi,\: \sum_{i=1}^r p_i + \frac{n-1}{2}\phi\right).
	\end{equation*}
	Since the form $J$ is a homogeneous degree 2 polynomial in $\mathbf{p}$ and $\nabla J$ is homogeneous of degree 1, the most general form for $\theta$ would be a first-order polynomial $\theta = t_1p_1 + \dots + t_r p_r + t_{r+1}\phi$ for some constants $t_1, \dots, t_{r+1}$. Then (\ref{factorisation}) becomes
	\begin{align*}
		-J &= \left[ \left(\sum_{i=1}^r c_i\left(\frac{2p_i}{d_i} + \phi\right)\right) + c_{r+1}\left( \sum_{i=1}^r p_i + \frac{n-1}{2}\phi \right)\right] \left[ \left(\sum_{i=1}^r t_ip_i\right) + t_{r+1}\phi\right]\\
		&=\left[ \left(\sum_{i=1}^r p_i\left( \frac{2c_i}{d_i} + c_{r+1}\right)\right) + \phi\left( \left(\sum_{i=1}^r c_i\right) + c_{r+1}\frac{(n-1)}{2}\right)\right]\left[ \left(\sum_{i=1}^r t_ip_i\right) + t_{r+1}\phi\right].
	\end{align*}
	In particular, since $r\geq 2$, we may compare the coefficients of $p_1^2$, $p_1p_2$ and $p_2^2$ on both sides of the equation to obtain
	\begin{align*}
		\frac{1}{d_1}&= \left(\frac{2c_1}{d_1} + c_{r+1}\right) t_1\\
		0 &= \left(\frac{2c_1}{d_1} + c_{r+1}\right) t_2 + \left( \frac{2c_2}{d_2} + c_{r+1}\right) t_1\\
		\frac{1}{d_2}&= \left( \frac{2c_2}{d_2} + c_{r+1}\right)t_2.
	\end{align*}
	Since $d_1, d_2\geq 1$, we have $t_1, t_2\neq 0$. Then we may divide through and substitute the first and third equations into the second to obtain $0=t_2^2/d_1 + t_1^2/d_2$, a contradiction.
\end{proof}

\medskip

\subsection{Generalised first integrals for the Bryant soliton}

In \cite{BetancourtDancerWang2016}, generalised first integrals for the $n=1$ and $n=4$ Bryant solitons are found. The latter is used to find an explicit formula for the 5-dimensional Bryant soliton. In this subsection, we show that the approach used in \cite{BetancourtDancerWang2016} to obtain a generalised first integral in the $n=4$ case does not work for any other value of $n$.

Recall that the hypersurfaces of the Bryant soliton are $n$-spheres viewed as the orbits of an $SO(n+1)$-action with isotropy subgroup $SO(n)$. The metric ansatz is
\begin{equation*}
	\bar{g} = dt^2 + h(t)^2g_{S^n},
\end{equation*}
where we set $h(t)^2=e^{q_1}$. We have $\tilde{\mathcal{W}}=\{(0,0),(-1,0)\}$ and $A_{(-1,0)}=n(n-1)$, as well as the extended vectors $\mathbf{d} = (n,-2)$, $\mathbf{q}=(q,u)$ and $\mathbf{p}=(p,\phi)$, and
\begin{equation*}
	J=-\left( \frac{p^2}{n} + p\phi + \frac{n-1}{4}\phi^2\right),
\end{equation*}
which factors as $J=(\mathbf{c}\cdot\nabla J)\theta$, where
\begin{equation}\label{ctheta}
	\mathbf{c}=\left( -\frac{1}{2}(n+\sqrt{n}), 1\right) \qquad \text{and} \qquad \theta = - \left( \frac{p}{\sqrt{n}} + \frac{\sqrt{n}-1}{2}\phi\right).
\end{equation}

\medskip

\textit{Remark.} We make an observation concerning the $n=4$ case which is treated in \cite{BetancourtDancerWang2016}, and we will use their notation. The provided first integral $F$ in \cite[(4.3)]{BetancourtDancerWang2016} satisfies $\{F,\mathcal{H}\}=0$. However, their choice of $\Gamma_0$ and $\Gamma_w$ is not unique, since it only has to satisfy $J\Gamma_v + \tau \theta^{s'} = J\Gamma_w + \rho\theta^s$. Their choice is $\Gamma_0 (= \Gamma_v) = -1$ and $\Gamma_w = 0$, but we may add some constant $b$ to both to get a new pair $\Gamma_0' = \Gamma_0 + b$, $\Gamma_w' = \Gamma_w + b$. As a result,
\begin{equation*}
	F_{\mathbf{c}+\mathbf{d}+\mathbf{z}}' = F_{\mathbf{c}+\mathbf{z}+\mathbf{d}} - bA_z \qquad \mathbf{z}\in\{0,\mathbf{w}\}
\end{equation*}
\begin{equation*}
	F'_{\mathbf{c}} = F_\mathbf{c} + bJ.
\end{equation*}
Since $A_0=E$ and $A_w = n(n-1)$, this leads to the difference in the first integrals being
\begin{equation*}
	F' - F = J e^{\mathbf{c}\cdot \mathbf{q}} - Ee^{(\mathbf{c}+\mathbf{d})\cdot\mathbf{q}} - n(n-1)e^{(\mathbf{c}+\mathbf{d}+\mathbf{w})\cdot \mathbf{q}}.
\end{equation*}
Computing the Poisson bracket, we get
\begin{equation*}
	\{F',\mathcal{H}\} = \{F'-F,\mathcal{H}\} = -\frac{b}{2} e^{-3q+u} (p+\phi)\mathcal{H},
\end{equation*}
so that $F'$ is also a generalised first integral.

\medskip

In situations where $\tilde{\mathcal{W}}$ contains two vectors, an approach of \cite[pp.9-10]{BetancourtDancerWang2016} is to find real-valued functions $\tau(\mathbf{p}), \rho(\mathbf{p})$ and constants $\Gamma_v, \Gamma_w$ such that 
\begin{equation}\label{twovectorcondition}
	\begin{gathered}
		(\mathbf{v}+\mathbf{d})\cdot \nabla\theta = -\frac{1}{s'} \qquad (\mathbf{w}+\mathbf{d})\cdot \nabla\theta = -\frac{1}{s}\\
		J\Gamma_v + \tau\theta^{s'} = J\Gamma_w + \rho\theta^s\\
		(\mathbf{v}+\mathbf{d})\cdot \nabla\tau = (\mathbf{w}+\mathbf{d})\cdot \nabla\rho = 0.
	\end{gathered}
\end{equation}
Then a solution to the recursion is
\begin{equation*}
	F_\mathbf{c} = J\Gamma_v + \tau \theta^{s'} (= J\Gamma_w + \rho\theta^s), \quad F_\mathbf{c} = \theta \psi_\mathbf{c}, \quad F_{\mathbf{c}+\mathbf{v}+\mathbf{d}} = -A_v \Gamma_v, \quad F_{\mathbf{c}+\mathbf{w}+\mathbf{d}} = -A_w \Gamma_w,
\end{equation*}
with all other terms zero. The generalised first integral is then given by $F=\sum_\mathbf{c} F_\mathbf{c} e^{\mathbf{c}\cdot\mathbf{q}}$.

In the case of the Bryant soliton, $\mathbf{v}=(0,0)$ and $\mathbf{w}=(-1,0)$. With the factorisation (\ref{ctheta}), one calculates \cite{BetancourtDancerWang2016}
\begin{equation*}
	\mathbf{d}\cdot\nabla\theta = -1, \qquad (\mathbf{d}+\mathbf{w})\cdot\nabla\theta = -1+\frac{1}{\sqrt{n}} = -\frac{\sqrt{n}-1}{\sqrt{n}},
\end{equation*}
implying that $s'=1$ and $s=\frac{\sqrt{n}}{\sqrt{n}-1}$. The condition $\mathbf{d}\cdot\nabla\tau = (n,-2)\cdot \nabla\tau =0$ is a first-order linear PDE. By abuse of notation, we will denote its solution by $\tau(p,\phi)=\tau(2p + n\phi)$ for a real-valued differentiable function $\tau(\cdot)$. Similarly, $(\mathbf{d}+\mathbf{w})\cdot\nabla\rho = 0$ implies that $\rho(p,\phi) = \rho(2p + (n-1)\phi)$. It only remains to check the middle equation of (\ref{twovectorcondition}), which is given by
\begin{equation*}
	J\Gamma_0 + \tau(2p+n\phi)\theta = J\Gamma_w + \rho(2p+(n-1)\phi)\theta^{\frac{\sqrt{n}}{\sqrt{n}-1}}.
\end{equation*}
Using the factorisation $J=(\mathbf{c}\cdot\nabla J)\theta = \left( \frac{p}{\sqrt{n}} + \frac{\sqrt{n}+1}{2}\phi\right)\theta$ and the coordinates $(x,y)$ defined by $x=2p+n\phi$ and $y=2p+(n-1)\phi$, we get
\begin{equation*}
	\frac{\Gamma_0-\Gamma_w}{2\sqrt{n}} \left((\sqrt{n}+1)x+y\right) + \tau(x) = \rho(y) \left(\frac{1}{2\sqrt{n}}\right)^{\frac{1}{\sqrt{n}-1}} \left( (\sqrt{n}-1)x-y\right)^\frac{1}{\sqrt{n}-1}.
\end{equation*}
Notice that either $\rho\equiv 0$, or $\rho(y)\neq 0$ for each $y$. Therefore we may divide by $\rho$ on both sides, then differentiate with respect to $x$. We obtain an equality of the form
\begin{equation*}
	F(x)G(y) = \left((\sqrt{n}-1)x - y\right)^\frac{2-\sqrt{n}}{\sqrt{n}-1}.
\end{equation*}
If $n\neq 4$, we take the $\frac{\sqrt{n}-1}{2-\sqrt{n}}$-th power of both sides, and see that $F,G$ are constant, and therefore that $\rho$ is constant, in which case we must have $n=4$, contradiction.

Therefore this approach can only obtain a generalised first integral for the Bryant soliton in fibre dimension $n=4$.

\bigskip
\textbf{Acknowledgements.} I am grateful to Andrew Dancer and Jason Lotay for valuable discussions, guidance and comments on the manuscript.

\bibliography{refssoliton}
\bibliographystyle{amsplain}
\end{document}